\documentclass[12pt,a4paper]{amsart}
\usepackage[latin1]{inputenc}
\usepackage[english]{babel}
\usepackage{amsmath,amstext,amsthm,amsfonts,amssymb,amscd}
\usepackage[dvips]{graphicx}
\usepackage{hyperref}

\usepackage{cleveref}
\usepackage{color}

\theoremstyle{plain}
\newtheorem{theorem}{Theorem}[subsection]
\newtheorem{proposition}{Proposition}[subsection]
\newtheorem{definition}{Definition}[subsection]
\newtheorem{remark}{Remark}[subsection]

\newtheorem{example}{Example}[subsection]
\newtheorem{claim}{Claim}[subsection]
\newtheorem{lemma}{Lemma}[subsection]
\newtheorem{corollary}{Corollary}[subsection]
\newtheorem{theorema}{Theorem}

\begin{document}
	
	\title[Open Zooming Systems: Uniqueness of Equilibrium States]{Open Zooming Systems: Uniqueness of Equilibrium States}
	
	\author[R. A. Bilbao]{Rafael A. Bilbao}
	\address{Rafael A. Bilbao, Universidad Pedag\'ogica y Tecnol\'ogica de Colombia, Avenida
		Central del Norte 39-115, Sede Central Tunja, Boyac\'a, 150003, Colombia. 
	}
	\email{rafael.alvarez@uptc.edu.co}

	\author[E. Santana]{Eduardo Santana}
	\address{Eduardo Santana, Universidade Federal de Alagoas, 57200-000 Penedo, Brazil}
	\email{jemsmath@gmail.com}


\date{\today}

\maketitle

\begin{abstract} 
We study open zooming systems and potentials with uniqueness of equilibrium states. The uniqueness is established for a certain class of zooming potentials when the map is topologically exact, including the null one. Also, with equilibrium stability, we prove that there exists a countable and an open sets of continuous potentials with uniqueness which are both dense in the set of continuous potentials with finiteness. The results here are related to the works \cite{BiS} and \cite{S} where finiteness and stability are studied.
\end{abstract}

\bigskip



\section{Introduction}
The theory of equilibrium states in dynamical systems was first developed by Sinai, Ruelle and Bowen in the sixties and seventies. 
It was based on applications of techniques of Statistical Mechanics to smooth dynamics. The classical theory of equilibrium states is developed for continuous maps. Given a continuous map $f: M \to M$ on a compact metric space $M$ and a continuous potential $\phi : M \to \mathbb{R}$, an  \textbf{\textit{equilibrium state}} is an 
invariant measure that satisfies a variational principle, that is, a measure $\mu$ such that
\begin{equation}\label{equilibrium}
	\displaystyle h_{\mu}(f) + \int \phi d\mu = \sup_{\eta \in \mathcal{M}_{f}(M)} \bigg{\{} h_{\eta}(f) + \int \phi d\eta \bigg{\}},
\end{equation}
where $\mathcal{M}_{f}(M)$ is the set of $f$-invariant probabilities on $M$ and $h_{\eta}(f)$ is the so-called metric entropy of $\eta$.

For measurable maps we define equilibrium states as follows.
Given a measurable map $f: M \to M$ on a compact metric space $M$ for which the set $\mathcal{M}_{f}(M)$ of invariant measures is non-empty (for example, the continuous maps) and a measurable potential $\phi : M \to \mathbb{R}$, we define the \textbf{\textit{pressure}} $P_{f}(\phi)$ as 
\begin{equation}\label{pressure}
	\displaystyle P_{f}(\phi) : = \sup_{\eta \in \mathcal{M}_{f}(M)} \bigg{\{} h_{\eta}(f) + \int \phi d\eta \bigg{\}}
\end{equation}
and an \textbf{\textit{equilibrium state}} is an invariant measure that attains the supremum, and it is defined analogously to the case for continuous maps (see equation \ref{equilibrium}).

In the context of uniform hyperbolicity, which includes uniformly expanding maps, equilibrium states do exist and are unique if the potential is H\"older continuous
and the map is transitive. In addition, the theory for finite shifts was developed and used to achieve the results for smooth dynamics.

Beyond uniform hyperbolicity, the theory is still far from complete. It was studied by several authors, including Bruin, Keller, Demers, Li, Rivera-Letelier, Iommi and Todd 
\cite{BDM, BK, BT,DT,IT1,IT2,LRL} for interval maps; Denker and Urbanski  \cite{DU} for rational maps; Leplaideur, Oliveira and Rios 
\cite{LOR} for partially hyperbolic horseshoes; Buzzi, Sarig and Yuri \cite{BS,Y}, for countable Markov shifts and for piecewise expanding maps in one and higher dimensions. 
For local diffeomorphisms with some kind of non-uniform expansion, there are results due to Oliveira \cite{O}; Arbieto, Matheus and Oliveira \cite{AMO};
Varandas and Viana \cite{VV}, all of whom proved the existence and uniqueness of equilibrium states for potentials with low oscillation. Also, for this type of map,
Ramos and Viana  \cite{RV}  proved it for the so-called  \textbf{\textit{hyperbolic potentials}}, which include these previous ones for the case of non-uniform expansion. The hyperbolicity of the potential is
characterized by the fact that the pressure emanates from the hyperbolic region. In most of these studies previously cited, the maps do not have the presence of critical sets. In \cite{BDM} and \cite{DT}, for example, the authors develop results for open interval maps with critical sets, but not for hyperbolic potentials and, recently, Alves, Oliveira and Santana proved the existence of equilibrium states for hyperbolic potentials, possibly with the presence of a critical set (see \cite{AOS}). We will see that the potentials considered in \cite{LRL}, which allow critical sets, are included in the potentials considered in \cite{AOS}. Here, we give an example of a class of hyperbolic potentials. It includes the null potential for Viana maps. We stress that one of the consequences in this paper is the existence and uniqueness of measures of maximal entropy for Viana maps and we observe that in \cite{ALP} the authors show, in particular, the existence of at most countably many ergodic measures of maximal entropy. In \cite{LOP} the authors prove the existence of at most one. Recently, in the work \cite{L} the author announced uniqueness for Viana maps and potentials with small variations. In \cite{PV} Pinheiro-Varandas obtain existence and uniqueness by using another technique than us for what they call \textbf{\textit{expanding potentials}}. They are also considered in \cite{AOS}, include the hyperbolic ones and are included in our \textbf{\textit{zooming potentials}}, introduced in \cite{S}. It is proved that the class of hyperbolic potentials is equivalent to the class of continuous zooming potentials and that they include the null one. It guarantees uniqueness of measure of maximal entropy, including for Viana maps. The uniqueness in general is established when the maps is \textbf{\textit{topologically exact}}.

For interval maps and the so-called \textbf{\textit{geometric potentials}} $\phi_{t} = -t\log \mid Df \mid$, equilibrium states were studied in  \cite{BT}, \cite{IT1}, \cite{IT2} and \cite{PS}. These results inspired ours, for measurable zooming systems on metric spaces, where the potential is $\phi_{t} = -t\log J_{\mu}f$ and $J_{\mu}f$ is a Jacobian of the reference measure and we call them \textbf{\textit{pseudo-geometric potentials}}, which are examples of measurable ones. All results are for open dynamics, which reduce to the corresponding closed dynamics if the hole is empty. We also mention the works \cite{DT2} and \cite{PU} for important comprehension of the theory of equilibrium states for open systems and Hölder and geometric potentials. 

%

We also have the study the \textbf{\textit{equilibrium stability}} for some open zooming systems in \cite{BiS}. Once the finiteness is obtained, the continuity of equilibrium states as studied in \cite{Ar} when we do not have uniqueness can be studied. In \cite{ARS} and \cite{BR} the authors study stability with uniqueness. Similar results of existence and uniqueness but  without stability can be seen in \cite{SSV}.

%
%

\bigskip

\section{Setup and Main Results}\label{Setup}

In this section, we recall some definitions and results given in \cite{S} and give some definitions and state our main results. 

\subsection{Zooming sets and measures}

For differentiable dynamical systems, hyperbolic times are a powerful tool to obtain a type of expansion in the context of non-uniform expansion. As we can find in \cite{AOS}, it can be generalized for systems considered in a metric space, also with exponential contractions. The zooming times generalizes it beyond the exponential
context. Details can be seen in \cite{Pi1}.

Let $f : M \to M$ be a measurable map defined on a connected, compact, separable metric space $M$.

\begin{definition}
	(Zooming contractions). A \textbf{\textit{zooming contraction}} is a sequence of functions $\alpha_{n}: [0,+\infty) \to [0,+\infty)$ such that
	
	\begin{itemize}
		\item $\alpha_{n}(r) < r, \text{for all} \, \, n \in \mathbb{N}, \text{for all} \, \, r>0.$
		
		\item $\alpha_{n}(r)<\alpha_{n}(s), \,\, if \,\, 0<r<s, \text{for all} \, \, n \in \mathbb{N}$.
		
		\item $\alpha_{m} \circ \alpha_{n}(r) \leq \alpha_{m+n}(r), \text{for all} \, \, r>0, \text{for all} \, \, m,n \in \mathbb{N}$.
		
		\item $\displaystyle \sup_{r \in (0,1)} \sum_{n=1}^{\infty}\alpha_{n}(r) < \infty$.
	\end{itemize}
	
\end{definition}

As defined in \cite{PV}, we call the contraction $(\alpha_{n})_{n}$ \textbf{\textit{exponential}} if $\alpha_{n}(r) = e^{-\lambda n} r$ for some $\lambda > 0$ and \textbf{\textit{Lipschitz}} if $\alpha_{n}(r) = a_{n} r$ with $0 \leq a_{n} < 1, a_{m}a_{n} \leq a_{m+n}$ and $\sum_{n=1}^{\infty} a_{n} < \infty$. In particular, every exponential contraction is Lipschitz. We can also have the example with $a_{n} = (n+b)^{-a}, a > 1, b>0$.

\begin{definition}\label{times}
	(Zooming times). Let $(\alpha_{n})_{n}$ be a zooming contraction and $\delta>0$. We say that $n \in \mathbb{N}$ is an $(\alpha,\delta)$\textbf{\textit{-zooming time}} for $p \in M$ if there exists 
	a neighbourhood $V_{n}(p)$ of $p$ such that
	
	\begin{itemize}
		\item $f^{n}$ sends $\overline{V_{n}(p)}$ homeomorphically onto $\overline{B_{\delta}(f^{n}(p))}$;
		
		\item $d(f^{j}(x),f^{j}(y)) \leq \alpha_{n - j}(d(f^{n}(x),f^{n}(y)))$ for every $x,y \in V_{n}(p)$ and every $0 \leq j < n$.
	\end{itemize}
	
	We call $B_{\delta}(f^{n}(p))$ a \textbf{\textit{zooming ball}} and $V_{n}(p)$ a  \textbf{\textit{zooming pre-ball}}.
\end{definition}

We denote by $Z_{n}(\alpha,\delta,f)$ the set of points in $M$ for which $n$ is an $(\alpha, \delta)$- zooming time.

\begin{definition}
	(Zooming measure) A $f$-non-singular finite measure $\mu$ defined on the Borel sets of M is called a \textbf{\textit{weak zooming measure}} if $\mu$ almost every point has 
	infinitely many $(\alpha, \delta)$-zooming times. A weak zooming measure is called a \textbf{\textit{zooming measure}} if
\end{definition}
\begin{equation}\label{frequency}
	\displaystyle \limsup_{n \to \infty} \frac{1}{n} \{1 \leq j \leq n \mid x \in Z_{j}(\alpha,\delta,f)\} > 0,
\end{equation}
$\mu$ almost every $x \in M$.

\begin{definition}
	(Zooming set) A forward invariant set $\Lambda \subset M$ (that is, $T(\Lambda) \subset \Lambda$) is called a \textbf{\textit{zooming set}} if the above inequality \ref{frequency} holds for every $x \in \Lambda$. 
\end{definition}
\begin{remark}\label{fixed}
	We stress that measure $\mu$ is not necessarily ergodic and the definition of the zooming set $\Lambda$ depends on the contraction $\alpha$ and the parameter $\delta$. Moreover, every forward invariant subset $\Lambda_{0} \subset \Lambda$ ($f(\Lambda_{0}) \subset \Lambda_{0}$) is also a zooming set if $\Lambda_{0}$ is full measure for $\mu$. We then need to previously fix a zooming set of reference $\Lambda$, which depends both on $\alpha$ and $\delta$.
\end{remark}

\begin{definition}
	(Bounded distortion) Given a measure $\mu$ with a Jacobian $J_{\mu}f$, we say that the measure has \textbf{\textit{bounded distortion}} if there exists $\rho > 0$ such that
	\[
	\bigg{|}\log \frac{J_{\mu}f^{n}(y)}{J_{\mu}f^{n}(z)} \bigg{|} \leq \rho d(f^{n}(y),f^{n}(z)),
	\]
	for every $y,z \in V_{n}(x)$, $\mu$-almost everywhere $x \in M$, for every zooming time $n$ of $x$.
\end{definition}

The map $f$ with an associated zooming measure with bounded distortion is called a \textbf{\textit{zooming system}}. We then denote a zooming system by $(f,M,\mu,\Lambda)$, where $\Lambda$ is previously fixed, as observed in Remark \ref{fixed}.

\subsection{Pressure, zooming potentials and equilibrium states} For measurable maps we recall the definition of an equilibrium state, given in the Introduction. Given a measurable map $f: M \to M$ on a compact metric space $M$ for which the set $\mathcal{M}_{f}(M)$ of $f$-invariant measures is non-empty (for example, the continuous maps) and a measurable potential $\phi : M \to \mathbb{R}$, we define the \textbf{\textit{pressure}} $P_{f}(\phi)$ as in equation \ref{pressure} and an  \textbf{\textit{equilibrium state}} is an invariant measure that attains the supremum as in equation \ref{equilibrium}.

Denote by $\mathcal{Z}(\Lambda)$ the set of invariant zooming measures supported on $\Lambda$ (in [\cite{Pi1}, Theorem C] it is proved that this set is nonempty under the same hypothesis as ours for the map, that is, a zooming measure with bounded distortion). 

We define a \textbf{\textit{zooming potential}} as a measurable potential $\phi:M \to \mathbb{R}$ such that
\[
\displaystyle  \sup_{\eta \in \mathcal{Z}(\Lambda)^{c}} \bigg{\{} h_{\eta}(f) + \int \phi d\eta \bigg{\}} < \sup_{\eta \in \mathcal{Z}(\Lambda)} \bigg{\{} h_{\eta}(f) + \int \phi d\eta \bigg{\}}.
\]
It is analogous to the definition of expanding potential that can be found in \cite{PV}.

Denote by $h(f)$ the pressure of the potential $\phi \equiv 0$ (or topological entropy), which we call simply \textbf{\textit{entropy}} of $f$, that is,
\[
\displaystyle h(f) := P_{f}(0)  = \sup_{\eta \in \mathcal{M}_{f}(M)} \bigg{\{} h_{\eta}(f) \bigg{\}}.
\]

\begin{example}\label{Base}
	As an example of a zooming potential, if we assume that a potential $\phi_{0} : M \to \mathbb{R}$ is such that
	\[
	\displaystyle  \sup_{\eta \in \mathcal{Z}(\Lambda)^{c}} \bigg{\{} \int \phi_{0} d\eta \bigg{\}} < \sup_{\eta \in \mathcal{Z}(\Lambda)} \bigg{\{} \int \phi_{0} d\eta \bigg{\}},
	\]
	there exists $t_{0} > 0$ such that
	\[
	h(f) < \sup_{\eta \in \mathcal{Z}(\Lambda)} \bigg{\{} \int t_{0}\phi_{0} d\eta \bigg{\}} - \sup_{\eta \in \mathcal{Z}(\Lambda)^{c}} \bigg{\{} \int t_{0}\phi_{0} d\eta \bigg{\}}.
	\]
	By taking $\phi:=t_{0}\phi_{0}$, we obtain
	\[
	\displaystyle  \sup_{\eta \in \mathcal{Z}(\Lambda)^{c}} \bigg{\{}h_{\eta}(f) + \int \phi d\eta \bigg{\}} \leq h(f) + \sup_{\eta \in \mathcal{Z}(\Lambda)^{c}} \bigg{\{} \int \phi d\eta \bigg{\}}< \sup_{\eta \in \mathcal{Z}(\Lambda)} \bigg{\{} \int \phi d\eta \bigg{\}} \leq \sup_{\eta \in \mathcal{Z}(\Lambda)} \bigg{\{}h_{\eta}(f) + \int \phi d\eta \bigg{\}}.
	\]
\end{example}

\begin{example} \label{Dirac}
	Let us suppose that there exists a fixed point $x_{0} \in \Lambda$ and a potential $\phi_{0} : M \to \mathbb{R}$ such that $\phi_{0}(x) < \phi_{0}(x_{0})$ for all $x \in M \backslash \{x_{0}\}$. It is clear that the Dirac probability $\delta_{x_{0}}$ supported on $x_{0}$ is a zooming measure. Moreover, if we assume that
	\[
	\displaystyle  \sup_{\eta \in \mathcal{Z}(\Lambda)^{c}} \bigg{\{} \int \phi_{0} d\eta \bigg{\}} < \int \phi_{0} d \delta_{x_{0}} = \sup_{\eta \in \mathcal{Z}(\Lambda)} \bigg{\{} \int \phi_{0} d\eta \bigg{\}},
	\]
	then, by the previous example, we can find $t_{0} > 0$ such that $\phi:= t_{0}\phi_{0}$ is a zooming potential.
\end{example}

\begin{example}\label{Periodic}
	To generalize the previous example, let us suppose that there exists a periodic point $x_{0} \in \Lambda$ with period $k \geq 1$ and a potential $\phi_{0} : M \to \mathbb{R}$ such that $\phi_{0}(x) < \phi_{0}(f^{i}(x_{0})), i=0,1,\dots,k-1$ for all $x \in M \backslash \{f^{i}(x_{0}) \mid i=0,1,\dots,k-1\}$. It is clear that the average of Dirac probabilities $\mu_{0}:=(1/k)\sum_{i=0}^{k-1}\delta_{f^{i}(x_{0})}$ supported on the orbit of $x_{0}$ is a zooming measure. Moreover, if we assume that
	\[
	\displaystyle  \sup_{\eta \in \mathcal{Z}(\Lambda)^{c}} \bigg{\{} \int \phi_{0} d\eta \bigg{\}} < \int \phi_{0} d \mu_{0} = \sup_{\eta \in \mathcal{Z}(\Lambda)} \bigg{\{} \int \phi_{0} d\eta \bigg{\}},
	\]
	then, we can find $t_{0} > 0$ such that $\phi:= t_{0}\phi_{0}$ is a zooming potential.
\end{example}

\begin{remark}
	We observe that in the previous examples we have for every invariant probability $\mu$
	\[
	\int \phi_{0} d\mu \leq \int \phi_{0} d\delta_{x_{0}} \,\, \text{and} \,\, \int \phi_{0} d\mu \leq \int \phi_{0} d\mu_{0}, \,\, \text{respectively}.
	\]
	What we assume is that the supremum of the integrals over all non zooming measures is strictly less than the integral with respect to the probabilities $\delta_{x_{0}}$ and $\mu_{0}$, respectively.
\end{remark}

For open systems, we consider another type of pressure and equilibrium states as follows.

\begin{definition}(Open pressure and open equilibrium states) \label{openpressure}
	Given an open system $(f,M,H)$ with a hole $H$, we define the \textbf{\textit{open pressure}} as
	\begin{equation}\label{open pressure}
		\displaystyle P_{f,H}(\phi) : = \sup_{\eta \in \mathcal{M}_{f}(M,H)} \bigg{\{} h_{\eta}(f) + \int \phi d\eta \bigg{\}},
	\end{equation}
	where $\mathcal{M}_{f}(M,H)$ is the set of $f$-invariant measures $\eta$ such that $\eta(H) = 0$. Moreover, we define an \textbf{\textit{open equilibrium state}} as an invariant measure $\mu \in \mathcal{M}_{f}(M,H)$ such that
	\begin{equation}\label{open equilibrium}
		\displaystyle h_{\mu}(f) + \int \phi d\mu = \sup_{\eta \in \mathcal{M}_{f}(M,H)} \bigg{\{} h_{\eta}(f) + \int \phi d\eta \bigg{\}},
	\end{equation} 	
\end{definition}

\begin{remark}
	We readily have that $P_{f,H}(\phi) \leq P_{f}(\phi)$. Also, in the case where the system is closed, that is, where $H = \emptyset$, the open pressure reduces to the pressure and the open equilibrium state reduces to the equilibrium state, both previously defined.
\end{remark}

\begin{remark}
	The definitions \ref{open pressure} and \ref{open equilibrium} are more appropriate for the open systems because they focus on measures which give full mass to the survivor set $M^{\infty}$. In fact, if a measure gives full mass to the survivor set $M^{\infty}$, then it gives null mass to the hole $H$. Also, since the survivor set $M^{\infty}$ is invariant, if an $f$-invariant measure gives null mass to the hole $H$, then it gives full mass to the survivor set $M^{\infty}$. The study of equilibrium states is now concentrated at the survivor set $M^{\infty}$.
\end{remark}

\begin{definition}
	(Backward separated map) We say that a map $f:M \to M$ is \textbf{\textit{backward separated}} if for every finite set $F \subset M$ we have
	\[
	\displaystyle d\Bigg{(}F, \bigcup_{j=1}^{n} f^{-j}(F) \backslash F \Bigg{)} > 0, \text{for all} \, \,\, n \geq 1.
	\]
\end{definition}
Observe that if $f$ is such that $\sup\{\# f^{-1} (x) \mid x \in M \}< \infty$, then $f$ is backward separated.


\begin{definition}(Open zooming systems) \label{open zooming}
	In order to make a zooming system $(f,M,\mu,\Lambda)$ \textbf{\textit{open}}, we take an open set $H \subset M$ as the hole and consider as the new zooming set the invariant set $\Lambda_{H} = \Lambda \cap M^{\infty}$ (if nonempty), where $M^{\infty}$ is the survivor set. We then take a measure $\mu_{H}$ such that $\mu_{H}(\Lambda_{H}) = 1$ as a new reference zooming measure and obtain the zooming system denoted by $(f,M,\mu_{H}, \Lambda_{H}, H)$. In order to have an open zooming system with Markov structure adapted to the hole, we consider the zooming system $(f,M,\mu,\Lambda_{H})$ and take a hole $H_{0} \subset H$ obtained by \cite{S}{Theorem A}. When we already have $\Lambda$ not dense in $M$, we can take $H$ disjoint from $\Lambda$ and we obtain $\Lambda \subset M^{\infty}$.
\end{definition}

Now, we state \cite{S}{Theorem B} on equilibrium states. With a quite general setup, we have existence and finiteness.

\begin{theorema}
	\label{B}
	Given a measurable open zooming system $f:M \to M$ which is backward separated, if the contraction $(\alpha_{n})_{n}$ satisfies $\alpha_{n}(r) \leq ar$ for some $a \in (0,1)$, every $n \in \mathbb{N}$ and every $r \in [0,+\infty)$ (Lipschitz, for example), with zooming set  $\Lambda$ (previously fixed) and hole $H$ given by Definition \ref{open zooming}. 
	
	\begin{itemize}
		
		\item For closed zooming systems $(H = \emptyset)$, if $\phi:M \to \mathbb{R}$ is a zooming potential with  finite pressure $P_{f}(\phi)$ and locally Hölder induced potential ($\phi$ Hölder, for example), then there are \textbf{\textit{finitely many}}  ergodic equilibrium states and they are zooming measures.  
		
		\item For open zooming systems, since $\Lambda_{H} \cap H = \emptyset$, we obtain the equilibrium states for the closed system giving full mass to the survivor set $M^{\infty}$. Then, in this case, the equilibrium states are open.
		
		\item There exists a zooming set $\Lambda_{\phi} \subset \Lambda$ with respect to which the potential $\phi$ is also zooming and has \textbf{\textit{uniqueness}}.
		
		\item The potential $\phi$ has \textbf{\textit{uniqueness}} if, in addition, for every forward invariant subset $\Lambda_{0} \subset \Lambda$  we have
		\[
		\displaystyle  \sup_{\eta \in \mathcal{Z}(\Lambda_{0})^{c}} \bigg{\{} h_{\eta}(f) + \int \phi d\eta \bigg{\}} \neq \sup_{\eta \in \mathcal{Z}(\Lambda_{0})} \bigg{\{} h_{\eta}(f) + \int \phi d\eta \bigg{\}}.
		\]
	\end{itemize}
\end{theorema}

\begin{remark}\label{remark} We observe that
	\begin{itemize}	
		\item Theorem \ref{B} includes a similiar result that can be seen in \cite{AOS} and the proof is along the same lines. Later on, we will see that hyperbolic potentials are equivalent to continuous zooming potentials and we can use \cite{S}{Theorem A} to obtain the same result for open non-uniformly expanding maps. Theorem \ref{B} reduces to closed system if $H = \emptyset$. Theorem \ref{B} is also similar to some of the main results of \cite{PV}. The main difference is the technique and we also include zooming systems with nonexponential Lipschitz contraction.
		
		\item The coding of the system by using the Markov structure and considering the system as closed gives finitely many ergodic equilibrium states which are zooming measures. Since the Markov structure is adapted to the hole $H$, as the zooming set $\Lambda$ is disjoint from $H$, we also obtain the property that the equilibrium states are supported on the survivor set $M^{\infty}$, since the measures are zooming (giving full mass to $\Lambda$ and, so, to $M^{\infty}$). As a consequence, they give null mass to the hole $H$ and we have $P_{f,H}(\phi) = P_{f}(\phi)$, which implies that the equilibrium states are open equilibrium states when the system is open and $\Lambda \cap H = \emptyset$. 
		
		
		\item If $\Lambda$ is dense in $M$ and there exists a not dense forward invariant subset $\Lambda_{0} \subset \Lambda$ such that $\mu(\Lambda \backslash \Lambda_{0}) = 0$, we can consider the open system with respect to the zooming set $\Lambda_{0}$.
		
		\item We always take a zooming set not dense in $M$ and disjoint from the hole to guarantee the existence of equilibrium states. It remains to investigate the case where the zooming set is not disjoint from the hole. A problem may occur when the closed system has uniqueness and the hole intersects its support. However, if there exists an equilibrium state for the closed system for which the hole does not intersect its support, we have it as an open equilibrium state.
	\end{itemize}
\end{remark}

\subsection{Pseudo-conformal measures and pseudo-geometric potentials}
Given a measure $\mu$ on $M$, its \textbf{\textit{Jacobian}} is a function $J_{\mu}f : M \to [0,+ \infty)$ such that 
\[
\mu(f(A)) = \int_{A} J_{\mu}f d\mu
\] 
for every $A$ \textbf{\textit{domain of injectivity}}, that is, a measurable set such that $f(A)$ is measurable and $f_{A}: A \to f(A)$ is a bijection.

The class of conformal measures is among the main measures we can choose in order to study the thermodynamic formalism of a given dynamical system $f:M \to M$. They are measures $\eta$ with a Jacobian $J_{\eta}f$ of the type $J_{\eta}f = e^{-\phi}$, where $\phi: M \to \mathbb{R}$ is a potential. It means that 
\[
\eta(f(A)) = \int_{A} e^{-\phi} d\eta,
\]
for every measurable domain of injectivity $A \subset M$. The potentials we will consider are the so-called \textbf{\textit{pseudo-geometric potentials}}, defined as follows
\[
\phi_{t}(x) =
\begin{cases}
	-t \log J_{\mu}f(x), & \text{if }   J_{\mu}f(x) \neq 0;\\
	0 , & \text{if }  J_{\mu}f(x) = 0.\\
\end{cases}
\]

We require that the Jacobian $J_{\mu}f$ is bounded above and the set of the points $x \in M$ where $J_{\mu}f(x) = 0$ has zero measure $\mu$. It means that
\[
\int e^{-\phi_{t}} d\mu = \int_{J_{\mu}f(x) \neq 0} e^{-\phi_{t}} d\mu =  \int_{J_{\mu}f(x) \neq 0} e^{t \log J_{\mu}f(x)} d\mu = \int (J_{\mu}f)^{t} d\mu
\]
and we call the measure $\mu$ \textbf{\textit{pseudo-conformal}}.

For the zooming reference measure $\mu$, Pinheiro showed in [\cite{Pi1}, Theorem C] that there are finitely many ergodic absolutely continuous measure with respect to $\mu$. We denote by $\mathbb{A}$ the set of such measures. Fix $\mu_{0} \in \mathbb{A}$ we know that $J_{\mu}f$ is also a Jacobian for $\mu_{0}$. From the definition of Jacobian, since $\mu_{0}$ is an invariant measure, it follows that $J_{\mu}f(x) \geq 1$, $\mu_{0} \,\, \text{a.e.} \,\, x \in M$. If fact, we have for domain of injectivity $A$
\[
\mu_{0}(A) \leq  \mu_{0}(f(A))=  \int_{A} J_{\mu}f d\mu_{0} \implies J_{\mu}f(x) \geq 1, \mu_{0} \,\, \text{a.e.} \,\, x \in A. 
\]
Since the zooming set $\Lambda$ has full measure $\mu_{0}$ and every zooming pre-ball is a domain of injectivity, we can cover a set of full measure with domains of injectivity. We conclude that $J_{\mu} f(x) \geq 1$ a.e. $x \in M$. We cannot have $J_{\mu} f(x) \equiv 1$ a.e. $x \in M$. So,
\[
\displaystyle \int \log J_{\mu} f d\mu_{0} > 0.
\]
We assume that $h(f) < \infty$ and let
\[
\displaystyle t_{0} : = \max_{\mu_{0} \in \mathbb{A}} \bigg{\{} \frac{h(f)}{-\int \log J_{\mu} f d\mu_{0}} \bigg{\}} \leq 0.
\]
The pressure $P_{f}(\phi_{t})$ is finite in our setting, since the potential $\phi_{t}$ is bounded above.

\begin{definition}
	We say that a map $f:M \to M$ is \textbf{\textit{topologically exact}} if for every open set $V \subset M$ there exists $k \in \mathbb{N}$ such that $f^{k}(V) = M$. 	
\end{definition}

\begin{theorema}
	\label{C}
	Given a measurable open zooming system $f:M \to M$ which is backward separated, if the contraction $(\alpha_{n})_{n}$ satisfies $\alpha_{n}(r) \leq ar$ for some $a \in (0,1)$, every $n \in \mathbb{N}$ and every $r \in [0,+\infty)$ (Lipschitz, for example), with zooming set  $\Lambda$ (previously fixed) and hole $H$ given by Definition \ref{open zooming}. 
	
	\begin{itemize}
		
		\item For $t < t_{0}$ the potential $\phi_{t} = -t \log J_{\mu}f$ is zooming and the induced potential is locally Hölder. 

		\item If the map is topologically exact, there exists a pseudo-conformal measure for the potential $\phi_{t} - P_{f}(\phi_{t})$, equivalent to some equilibrium state.

		%
		%
		
	\end{itemize} 
\end{theorema}

\begin{remark}
	Here it holds the same results and explanation as Remark \ref{remark} for Theorem \ref{B}. Moreover, we observe that a similar result can be found in \cite{IT1} for conformal measures and geometric potentials in the context of one-dimensional dynamics. We extend it to the context of metric spaces and open systems, with the proof along the same lines. The main ingredient here is the Markov structure obtained in our Theorem \cite{S}[Theorem A] and the theory of O. Sarig for symbolic dynamics.
\end{remark}

\subsection{Equilibrium stability} Theorem \ref{B} gives equilibrium states for certain types of measurable maps and potentials. Here, we define \textbf{\textit{equilibrium stability}} when both map and potential are continuous.

Let $\mathcal{CM}$ be the family of continuous maps $f: M \to M$ and $\mathcal{CP}$ the family of continuous potentials $\phi : M \to \mathbb{R}$. We say that a family $\mathcal{F} \subset \mathcal{CM} \times \mathcal{CP}$ for which every pair $(f,\phi) \in \mathcal{F}$ has existence of equilibrium states is  \textbf{\textit{equilibrium stable}} if we have continuity of the equilibrium states. It means that if the pair $(f_{n},\phi_{n}), n \geq 1$ has the measure $\mu_{n}$ as an equilibrium state and $f_{n} \to f_{0}$ and $\phi_{n} \to \phi_{0}$, then every accumulation point $\mu_{0}$ of $\mu_{n}$ is an equilibrium state for the pair $(f_{0},\phi_{0})$. The topology on $\mathcal{CM}$ is the $C^{0}$ one, on $\mathcal{CP}$ is also the $C^{0}$ topology and on the measures is the weak-$*$ topology. The topology on $\mathcal{CM} \times \mathcal{CP}$ is the product topology. 

For open systems we have $(f_{n}, \phi_{n}), n \geq 0$ with hole $H_{n}$ and we require that
\[
H_{0} = M\backslash \overline{\bigcup_{n=1}^{\infty} (M\backslash H_{n})}.
\] 
It is compatible with the convergence $\mu_{n} \to \mu_{0}$ because for every $n \geq 1$
\[
\mu_{n}(H_{n}) = 0 \implies \mu_{n}(M\backslash H_{n}) = 1 \implies \mu_{n}\Bigg{(} \overline{\bigcup_{i=1}^{\infty} (M\backslash H_{i})} \Bigg{)} = 1 \implies \mu_{n}(H_{0}) = 0.
\]
For open sets $A \subset M$ we have the following inequality:
\[
\mu_{0}(A) \leq  \liminf_{n \to \infty} \mu_{n}(A) \implies 0 \leq \mu_{0}(H_{0}) \leq  \liminf_{n \to \infty} \mu_{n}(H_{0}) = 0 \implies \mu_{0}(H_{0}) = 0.
\]
Moreover, let $\Lambda_{n}$ be the zooming set of $f_{n}, n 
\geq 0$. We take the new zooming set disjoint from the hole $H_{n}$ given by $\Lambda_{H_{n}}= \Lambda_{n} \cap M_{n}^{\infty}$, where $M_{n}^{\infty}$ is the survivor for the system $f_{n}$ with hole $H_{n}$. So, we can use either Pinheiro's result in \cite{Pi1}[Theorem D] or our Theorem \cite{S}[Theorem A] to obtain finitely many ergodic equilibrium states, each one giving null mass to the hole. So, they all are open equilibrium states as Theorem in \ref{B}. In other words, for every $n \geq 0$ we obtain open equilibrium states $\mu_{n}$, that is, 
\[
\displaystyle h_{\mu_{n}}(f_{n}) + \int \phi_{n} d\mu_{n} = \sup_{\eta \in \mathcal{M}_{f_{n}}(M,H_{n})} \bigg{\{} h_{\eta}(f_{n}) + \int \phi_{n} d\eta \bigg{\}} = P_{f_{n},H_{n}}(\phi_{n}),
\]
where $\mathcal{M}_{f_{n}}(M,H_{n})$ is the set of $f_{n}$-invariant measures giving null mass to the hole $H_{n}$.

By Theorem \ref{B} we have that the following family has finiteness of equilibrium states:
\[
\mathcal{FZ} = \{(f,\phi) \mid f,\phi \,\, \text{are both zooming  and} \,\, \phi \,\, \text{has induced potential locally H\"older} \}.
\]
\[
\mathcal{FZ}_{H} = \{(f,\phi) \mid f,\phi \,\, \text{are both zooming  and} \,\, \phi \,\, \text{has induced potential locally H\"older, hole} \,\, H_{f} \}.
\]
We then establish the following result:

\begin{theorema}
	\label{D}
	$\mathcal{FZ}$ and $\mathcal{FZ}_{H}$ are equilibrium stable. 
\end{theorema}

\begin{remark}
	This control over the holes guarantees compatibily over the convergence of equilibrium states because every equilibrium state must give null mass to the hole.
\end{remark}

With Theorem \ref{C} and stability, we can obtain the following theorem.

\begin{theorema}\label{F}
	\begin{itemize} In the context of Theorems \ref{B} and \ref{C}, we have
		\item If the map is topologically exact and $\phi$ is a zooming potential with locally Hölder induced potential, then the existence of pseudo-conformal measures imply uniqueness of equilibrium state.
		
		\item There exist a countable and an open sets of continuous potentials with uniqueness which are both dense in the set of continuous potentials with finiteness.
	\end{itemize}
\end{theorema}

\section{Uniqueness}\label{Uniqueness}

To obtain uniqueness in Theorem \ref{B}, if the potential is zooming for a pair of forward invariant subsets $\Lambda_{1}, \Lambda_{2} \subset \Lambda$ we have $\Lambda_{1}\cap \Lambda_{2} \neq \emptyset$ and suppose that there exist at least two ergodic zooming equilibrium states $\mu_{1}, \dots, \mu_{k}$. There exist pairwise disjoint invariant subsets $P_{1},\dots, P_{k}$ such that $\mu_{i}(P_{i}) = 1$ for every $i \leq k$. We can take $\Lambda_{i} = \Lambda \cap P_{i}$ to obtain $\mu_{i}$ as unique ergodic zooming equilibrium state. It means that $\mu_{1} = \dots = \mu_{k}$ because $\mu(\Lambda_{i}) = 1$ for every $i \leq k$ implies that $P_{i} \cap P_{j} \neq \emptyset$. The uniqueness is established.

Now, we prove that we always have $\Lambda_{1}\cap \Lambda_{2} \neq \emptyset$ in this case. In fact, otherwise, we would have
\[
\sup_{\eta \in \mathcal{Z}(\Lambda_{1})^{c}} \bigg{\{} h_{\eta}(f) + \int \phi d\eta \bigg{\}} < \sup_{\eta \in \mathcal{Z}(\Lambda_{1})} \bigg{\{} h_{\eta}(f) + \int \phi d\eta\bigg{\}} = h_{\mu_{1}}(f) + \int \phi d\mu_{1}= P_{f}(\phi),
\]
and
\[
\sup_{\eta \in \mathcal{Z}(\Lambda_{2})^{c}} \bigg{\{} h_{\eta}(f) + \int \phi d\eta \bigg{\}} < \sup_{\eta \in \mathcal{Z}(\Lambda_{2})} \bigg{\{} h_{\eta}(f) + \int \phi d\eta \bigg{\}} = h_{\mu_{2}}(f) + \int \phi d\mu_{2} = P_{f}(\phi).
\]
We must have $\Lambda_{1} \cap \Lambda_{2} \neq \emptyset$ because, otherwise, we would have $\mu_{1}(\Lambda_{2}) = 0$ and
\[
P_{f}(\phi) = h_{\mu_{1}}(f) + \int \phi d\mu_{1} \leq \sup_{\eta \in \mathcal{Z}(\Lambda_{2})^{c}} \bigg{\{} h_{\eta}(f) + \int \phi d\eta \bigg{\}} < 
\]
\[
\sup_{\eta \in \mathcal{Z}(\Lambda_{2})} \bigg{\{} h_{\eta}(f) + \int \phi d\eta \bigg{\}} = h_{\mu_{2}}(f) + \int \phi d\mu_{2}= P_{f}(\phi),
\]
which is a contradiction. Finally, if we have
\[
\sup_{\eta \in \mathcal{Z}(\Lambda_{0})^{c}} \bigg{\{} h_{\eta}(f) + \int \phi d\eta \bigg{\}} > \sup_{\eta \in \mathcal{Z}(\Lambda_{0})} \bigg{\{} h_{\eta}(f) + \int \phi d\eta \bigg{\}},
\]
for a certain forward invariant subset $\Lambda_{0} \subset \Lambda$, then we cannot guarantee existence of equlibrium states because the equilibrium states are all zooming measures. It proves uniqueness in Theorem \ref{B}.

\begin{proposition} \label{equivalent}
	Let $\phi_{t}$ be a pseudo-geometric potential. If $\nu$ is a pseudo-conformal measure, then it is equivalent to some equilibrium state $\mu$. Moreover, they have full support and every forward invariant subset $B \subset M$ such that $\nu(B) > 0$ is dense in $M$.
\end{proposition}
\begin{proof}
	Let $\overline{\mu}$ and $\overline{\nu}$ the corresponding lifts to the shift of $\mu$ and $\nu$, respectively. Theorem \cite{S}[Theorem 5.3.3] guarantees the existence of a positive function $h$ such that $d\overline{\mu} = h d\overline{\nu}$. Also, if $Q_{1}, \dots, Q_{k}, \dots$ are the elements of the partition $\mathcal{Q}$ for the inducing scheme $(F,\mathcal{Q})$ we have for a measurable subset $A \subset M$,
	\[
	\mu(A) = \sum_{k=1}^{\infty}
	\sum_{j=0}^{R_{k} - 1} \overline{\mu}(f^{-j}(A) \cap Q_{k}).
	\]
	So, if $\mu(A) > 0$ it implies that $\overline{\mu}(f^{-j}(A) \cap Q_{k}) > 0$ and $\overline{\nu}(f^{-j}(A) \cap Q_{k}) > 0$ for some $j, k \in \mathbb{N}$ and we obtain $\nu(A) > 0$. It means that $\mu$ is absolutely continuous with respect to $\nu$.
	Also, if $\nu(A) > 0$ it implies that $\overline{\nu}(f^{-j}(A) \cap Q_{k}) > 0$ and $\overline{\mu}(f^{-j}(A) \cap Q_{k}) > 0$ for some $j, k \in \mathbb{N}$ and we obtain $\mu(A) > 0$. It means that $\nu$ is absolutely continuous with respect to $\mu$. Then, $\mu$ and $\nu$ are equivalent.
	
	Once the map is topologically exact, given an open set $A \subset M$ we have $M = f^{m}(A)$ for some $m \in \mathbb{N}$ and
	\[
	0 < \nu(f^{m}(A)) \leq \int_{A} (J_{\mu}f^{m})^{t} d\nu \implies \nu(A) > 0,
	\] 
	and $\nu$ has full support. Also, $\mu$ has full support for every equilibrium state $\mu$ of the pseudo-geometric potential $\phi_{t}$.
	
	Let $B \subset M$ be a forward invariant subset such that $\nu(B) > 0$. It means that $\mu(B) > 0$ and then $\mu(B) = 1$ because $\mu$ is ergodic. Since every full measure subset is dense in the support, we have $B$ dense in $M$.	
\end{proof}

For uniqueness in Theorem \ref{F}, by Theorem \ref{B}, since the  potential $\phi$ is zooming with locally Hölder induced potential, we have finiteness: $\mu_{1}, \dots, \mu_{k}$. Take a subset $\Lambda \supset \Lambda_{1} \cup \dots \cup \Lambda_{k}$ where, for each $i \leq k$, $\Lambda_{i}$ is forward invariant such that $\mu_{i}(\Lambda_{i}) = 1$ and they are pairwise disjoint. As the map is topologically exact, by Theorem \ref{C}, there exists a pseudo-conformal measure $\nu$ and if 
\[
0 <  \nu(\Lambda_{1} \cup \dots \cup \Lambda_{k}) \leq \nu(\Lambda_{1}) + \dots + \nu(\Lambda_{k}),
\]
then, $\nu(\Lambda_{i}) > 0$ for some $i \leq k$ and also $\mu(\Lambda_{i}) > 0$, which implies $\mu(\Lambda_{i}) = 1$, where $\mu$ is an ergodic equilibrium state for the pseudo-geometric potential $\phi_{t}$ equivalent to $\nu$ by Proposition \ref{equivalent}. It means that for at most one $i \leq k$ we have $\nu(\Lambda_{i}) > 0$ because they are pairwise disjoint.

If $\Lambda \supset \Lambda_{1}' \cup \dots \cup \Lambda_{k}'$ where, for each $i \leq k$, $\Lambda_{i}'$ is forward invariant such that $\mu_{i}(\Lambda_{i}') = 1$ and they are pairwise disjoint, such that
\[
0 <  \nu(\Lambda_{1}' \cup \dots \cup \Lambda_{k}'),
\]
then $\mu(\Lambda_{j}')=1$ for some $j \leq k$ and it implies $\mu(\Lambda_{i} \cap \Lambda_{j}')=1$. But also, we obtain the following union $\Lambda \supset (\Lambda_{1} \cap \Lambda_{1}') \cup \dots \cup (\Lambda_{k} \cap \Lambda_{k}')$ such that $\mu_{i}(\Lambda_{i} \cap \Lambda_{i}') = 1$ and
\[
0 <  \nu((\Lambda_{1} \cap \Lambda_{1}') \cup \dots \cup (\Lambda_{k} \cap \Lambda_{k}')),
\]
and it implies either $\mu(\Lambda_{i} \cap \Lambda_{i}') = 1$ or $\mu(\Lambda_{j} \cap \Lambda_{j}') = 1$. But then $\mu(\Lambda_{i} \cap \Lambda_{j}) = 1$ and $\Lambda_{i} \cap \Lambda_{j} \neq \emptyset$, which is a contradiction, unless $j = i$.

Take the following decomposition $\Lambda = \Lambda_{1} \cup \dots \cup (P_{i} \cap \Lambda) \cup \dots \cup \Lambda_{k}$, where $\mu_{i}(\Lambda_{j}) = 1$ for $j \neq i$ and the subsets are pairwise disjoint. By the above discussion, we obtain $\mu(P_{i} \cap \Lambda) =1$ and so $\mu = \mu_{i}$ because both measures are ergodic.

We conclude that if a potential $\phi$ does not have uniqueness, so it has as equilibrium states all the equilibrium states of each pseudo-geometric potential $\phi_{t}$. It is a contradiction because all pseudo-geometric potentials cannot have commom equilibrium states. In fact, if $\eta$ were equilibrium state for  $\phi_{t}$ and every $t < t_{0}$, then for every invariant measure $\lambda$
\[
h_{\eta}(f) - t\int \log J_{\mu}f d\eta \geq h_{\lambda}(f) - t\int \log J_{\mu}f  d\lambda \implies \frac{h_{\eta}(f) - h_{\lambda}(f)}{\int \log J_{\mu}f d\lambda - \int \log J_{\mu}f d\eta} \geq -t,
\]
taking $\int \log J_{\mu}f d\lambda > \int \log J_{\mu}f d\eta$ and it is a contradiction if $t \to -\infty$. So, the potential $\phi$ must have uniqueness. Otherwise, it would have infintely many equilibrium states. The uniqueness is established.

\begin{remark}
	In order to guarantee that such a measure $\lambda$ exists  such that $\int \log J_{\mu}f d\lambda > \int \log J_{\mu}f d\eta$ we can take the Jacobian attaining its unique maximum at a fixed point $x_{0} \in M$ and we have for $\lambda = \delta_{x_{0}}$ (Dirac measure at $x_{0}$)
	\[
	\int \log J_{\mu}f d\eta < \max \log J_{\mu}f = \int \log J_{\mu}f d\delta_{x_{0}} = \int \log J_{\mu}f d\lambda.
	\]
	Otherwise we would have $\eta = \delta_{x_{0}}$.
\end{remark}
\section{Hyperbolic Potentials}\label{Hyperbolic}

In this section, we extend the notion of hyperbolic potentials which appears in \cite{AOS} for exponential contractions and for continuous maps to the general case of $f:M \to M$ being a zooming map and  the contraction $(\alpha_{n})_{n}$ satisfying $\alpha_{n}(r) \leq ar$ for some $a \in (0,1)$, every $n \in \mathbb{N}$ and every $r \in [0,+\infty)$ (Lipschitz contractions, for example). Also, we show that hyperbolic potentials are \textbf{equivalent} to continuous zooming potentials and give an example of a class of hyperbolic potentials. By proving that the null potential is zooming we obtain, in particular, the existence and uniqueness of measures of maximal entropy for zooming maps with general contractions, answering this question for the important class of maps known as Viana maps.

We stress the fact that in the case of a zooming map with non dense zooming set we can consider this section for the respective open zooming system.

We begin by recalling what we mean by relative pressure, which is a notion of pressure for non-compact sets, as the case of zooming sets in general.

\subsection{Topological pressure}

We recall the definition of relative pressure for non-compact sets by dynamical balls, as it is given in \cite{ARS}.
Let $M$ be a compact metric space. Consider $f:M \to M$ and $\phi: M \to \mathbb{R}$. Given $\delta > 0$, $n \in \mathbb{N}$ and $x \in M$, we define the 
\textbf{\textit{dynamical ball}} $B_{\delta}(x,n)$ as 
\[
B_{\delta}(x,n): = \{y \in M |d(f^{i}(x),f^{i}(y)) < \delta, \,\, \text{for} \,\, 0 \leq i \leq n\}.
\]
Consider for each $N \in \mathbb{N}$, the set
\[
\mathcal{F}_{N} = \left\{B_{\delta}(x,n) |x \in M, n \geq N\right\}.
\]
Given $\Lambda \subset M$, denote by $\mathcal{F}_{N}(\Lambda)$ the finite or countable families of elements in $\mathcal{F}_{N}$ that cover $\Lambda$. Define for $n \in \mathbb{N}$
\[
S_{n}\phi(x) = \phi(x) + \phi(f(x)) + \dots + \phi(f^{n-1}(x)).
\]
and
\[
\displaystyle R_{n,\delta}\phi(x) = \sup_{y \in B_{\delta}(x,n)} S_{n}\phi(y).
\]
Given a $f$-invariant set $\Lambda \subset M$, not necessarily compact, define for each $\gamma > 0$
\[
\displaystyle m_{f}(\phi, \Lambda, \delta, N, \gamma) = \inf_{\mathcal{U} \in \mathcal{F}_{N}(\Lambda)} \left\{ \sum_{B_{\delta}(y,n) \in \mathcal{U}}
e^{-\gamma n + R_{n,\delta}\phi(y)} \right\}.
\]
Define also
\[
\displaystyle m_{f}(\phi, \Lambda, \delta, \gamma) = \lim_{N \to + \infty} m_{f}(\phi, \Lambda, \delta, N, \gamma).
\]
and
\[
P_{f}(\phi, \Lambda, \delta) = \inf \{\gamma > 0 |m_{f}(\phi, \Lambda, \delta, \gamma) = 0\}.
\]
Finally, define the \textbf{\textit{relative pressure}} of $\phi$ on $\Lambda$ as
\[
P_{f}(\phi,\Lambda) = \lim_{\delta \to 0} P_{f}(\phi, \Lambda, \delta).
\]
The \textbf{\textit{topological pressure}} of $\phi$ is, by definition, $P_{f}(\phi) = P_{f}(\phi, M)$ and satisfies
\begin{eqnarray}
	\label{Pressures}
	P_{f}(\phi) = \sup \{P_{f}(\phi,\Lambda), P_{f}(\phi,\Lambda^{c})\}
\end{eqnarray}
where $\Lambda^{c}$ denotes the complement of $\Lambda$ on $M$. We refer the reader to \cite{Pe2} for the proof of~\eqref{Pressures} and for additional properties of the pressure.
See also \cite{W} for a proof of the fact that
\[
\displaystyle P_{f}(\phi) = \sup_{\mu \in \mathcal{M}_{f}(M)} \bigg{\{} h_{\mu}(f) + \int \phi d \mu \bigg{\}}.
\]

\subsection{Hyperbolic potentials}

Given a continuous zooming map $f:M \to M$ with general contractions, we say that a   continuous function $\phi : M \to \mathbb{R}$ is a \emph{hyperbolic potential} if the topological pressure $P_{f}(\phi)$ is located on the zooming set $\Lambda$, i.e.
\[
P_{f}(\phi,\Lambda^{c}) < P_{f}(\phi).
\]
This notion is extends the notion of hyperbolic potential in \cite{RV}, since they define the expanding set from an average and, later on, they prove the property of expansion with neighbourhoods by proving a distortion control in Lemma 3.6 and that the expanding set is a zooming set.

In \cite{IRL} I. Inoquio-Renteria and J. Rivera-Letelier use the term hyperbolic potential for the first time. As in \cite{LRL}, where H. Li and J. Rivera-Letelier consider other type of hyperbolic potentials for one-dimensinal dynamics. In their context, $\phi$ is a hyperbolic potential if
\begin{equation}\label{Li-Rivera}
	\displaystyle \sup_{\mu \in \mathcal{M}_{f}(M)} \int \phi d\mu < P_{f}(\phi).	
\end{equation}
We claim that these type of hyperbolic potentials are zooming. When the map is one-dimensional, a measure being zooming (or expanding) means that the Lyapunov exponent is positive. Otherwise, it is negative or zero. By Ruelle's inequality, we obtain that the entropy is negative or zero. In this case, for a measure $\mu$ that is not zooming, we obtain
\[
\displaystyle \sup_{\mu \in \mathcal{Z}(\Lambda)^{c}} \bigg{\{} h_{\mu}(f) + \int \phi d \mu \bigg{\}} \leq \sup_{\mu \in \mathcal{Z}(\Lambda)^{c}}\bigg{\{}\int \phi d \mu \bigg{\}} \leq 
\]
\[
\sup_{\mu \in \mathcal{M}_{f}(M)} \int \phi d\mu < P_{f}(\phi)=\sup_{\mu \in \mathcal{Z}(\Lambda)} \bigg{\{} h_{\mu}(f) + \int \phi d \mu \bigg{\}}.	
\]
It means that this type of potential is zooming (and hyperbolic, as we will see in this section) as defined above and we can use our Theorem \ref{B} to obtain finitely many ergodic equilibrium states which are zooming measures.

Another type of hyperbolic potentials considered, is taking $\phi$ such that
\begin{equation}\label{Przytycki-Rivera}
	\sup \phi < P_{f}(\phi).
\end{equation}
We can easily see that condition \ref{Przytycki-Rivera} implies \ref{Li-Rivera}.
We observe that every hyperbolic potential in our context is zooming, as we can see in the following proposition.

\begin{proposition}\label{HZ}
	Let $\phi$ be  a hyperbolic potential. If $\mu$ is an ergodic probability measure such that $h_{\mu}(f) + \int \phi d\mu > P_{f}(\phi,\Lambda^{c})$, then $\mu(\Lambda)=1$.
\end{proposition}

%

The main result of this section is the following theorem, which establishes the equivalence between hyperbolic and zooming potentials.

\begin{theorem}\label{HYPERZOOM}
	Let $f:M \to M$  a continuous zooming map and  the contraction $(\alpha_{n})_{n}$ satisfying $\alpha_{n}(r) \leq ar$ for some $a \in (0,1)$, every $n \in \mathbb{N}$ and every $r \in [0,+\infty)$ (Lipschitz contractions, for example)  and $\phi : M \to \mathbb{R}$ a continuous potential. Then $\phi$ is a hyperbolic potential if, and only if, it is a zooming potential.
\end{theorem}

We divide the proof of Theorem \ref{HYPERZOOM} into some Lemmas. The first Lemma proves that both the sets of hyperbolic and continuous zooming potentials are open in the topology of the supremum norm.

\begin{lemma} \label{OPEN}
	Let $f:M \to M$ a continuous zooming map and the contraction $(\alpha_{n})_{n}$ satisfying $\alpha_{n}(r) \leq ar$ for some $a \in (0,1)$, every $n \in \mathbb{N}$ and every $r \in [0,+\infty)$ (Lipschitz contractions, for example). Denote by $\mathcal{HP}$ the set of hyperbolic potentials and $\mathcal{ZP}$ the set of continuous zooming potentials. We have that both sets $\mathcal{HP}$ and $\mathcal{ZP}$ are open in the topology of the supremum norm $\parallel \cdot \parallel_{\infty}$. 
\end{lemma}

We observe that Proposition \ref{HZ} guarantees that $\mathcal{HP} \subset \mathcal{ZP}$. In order to prove Theorem \ref{HYPERZOOM} it remains to show that $\mathcal{ZP} \subset \mathcal{HP}$. In the next Lemma, we show this by showing that $\mathcal{ZP} \subset \overline{\mathcal{HP}}$ and using that $\mathcal{HP} \subset \mathcal{ZP}$ are open sets.

\begin{lemma}\label{ZOOMHYPER}
	With the notation of Lemma \ref{OPEN}, we have that $\mathcal{ZP} \subset \overline{\mathcal{HP}}$. 
\end{lemma}

\begin{theorem}\label{ZER}
	Let $f:M \to M$ a continuous zooming map and  the contraction $(\alpha_{n})_{n}$ satisfying $\alpha_{n}(r) \leq ar$ for some $a \in (0,1)$, every $n \in \mathbb{N}$ and every $r \in [0,+\infty)$ (Lipschitz contractions, for example) such that the topological entropy is positive ($h(f) > 0$). If there exists a zooming potential $\phi_{0}$ with locally Hölder induced potential, then the null potential $\phi \equiv 0$ is  zooming (and also hyperbolic). In particular, by Theorem \ref{B} there exist finitely many ergodic measures of maximal entropy which are zooming measures.
\end{theorem}

\begin{corollary}\label{BIR}
	Let $\phi: M \to \mathbb{N}$ with its Birkhoff sums uniformly bounded, that is, there exists $r > 0$ such that
	\[
	|S_{n}\phi(x)| < r, \text{for all} \, \, n \in \mathbb{N}, \text{for all} \, \, x \in M.
	\]
	Then, $\phi$ is a zooming (and hyperbolic) potential.
\end{corollary}
%

\begin{corollary}\label{mme}
	Under the conditions of Theorem \ref{ZER} and, in addition, if the map is topologically exact, there exists a unique measure of maximal entropy. 
\end{corollary}
A very important class of zooming maps is the class of Viana maps, defined in  section \ref{Examples}. The problem concerning the existence and uniqueness of the measure of maximal entropy has been studied for several authors. In \cite{ALP} the authors prove that there exist at most countably many of them. In \cite{AOS} the authors prove existence and finiteness. A proof of existence and uniqueness is announced in \cite{PV} and in \cite{L}with different approaches. We obtain it as a corollary of Theorem \ref{ZER}.
\begin{corollary}
	Let $f:S^{1} \times I \to S^{1} \times I$ be a Viana map. There exists a unique measure of maximal entropy for $f$.
\end{corollary}

\begin{remark}
	In \cite[Proposition 12.2]{ALP} the authors establish that for Viana maps we have the following result, among others: if $\mu$ is an $f$-invariant measure such that $h_{\mu}(f) \geq h_{SRB}(f)$ where $SRB$ denotes the unique SRB measure for Viana maps, then the measure $\mu$ is hyperbolic. It implies that
	\[
	\sup_{\nu \in \mathcal{Z}(\Lambda)^{c}}\{h_{\nu}(f)\} \leq \sup_{\mu \in \mathcal{Z}(\Lambda)}\{h_{\mu}(f)\} = h(f).
	\]
	If the inequality is strict, it means that the null potential is zooming. Otherwise, we still can take a sequence of zooming measures $\mu_{n}$ such that $h_{\mu_{n}}(f) \to h(f)$ and the proof to find equilibrium states proceeds analogously. We then find uniqueness of the measure of maximal entropy in any case.
\end{remark}

With Theorem \ref{ZER} and the next lemma, we can see that the constant potentials are all hyperbolic (and zooming).

\begin{lemma} 
	We have that $P_{\Lambda}(\phi + c) = P_{\Lambda}(\phi) + c$, for all potential $\phi$ and constant $c \in \mathbb{R}$. 
\end{lemma}

%

The previous lemma also shows that if $\phi$ is a hyperbolic (and zooming) potential, so is $\phi + c, \text{for all} \, \, c \in \mathbb{R}$. We can also obtain the following lemma.

\begin{lemma}
	If $\phi \leq \psi$, then $P_{\Lambda}(\phi) \leq P_{\Lambda}(\psi)$. 
\end{lemma}

With the previous lemma we can obtain the following examples.

\begin{example}
	Let $\varphi: M \to \mathbb{R}$ be a hyperbolic potential and $\phi:M \to \mathbb{R}$ such that
	\[
	\max \phi - \min \phi < P_{\Lambda}(\varphi) - P_{\Lambda^{c}}(\varphi). 
	\]
	It implies that
	\[
	P_{\Lambda^{c}}(\varphi + \phi)  \leq  P_{\Lambda^{c}}(\varphi + \max \phi) = P_{\Lambda^{c}}(\varphi) + \max \phi <
	\]
	\[
	< P_{\Lambda}(\varphi) + \min \phi = P_{\Lambda}(\varphi + \min \phi) \leq P_{\Lambda}(\varphi + \phi). 
	\]
	So, $\varphi + \phi$ is a hyperbolic potential.
	
	If $|t|\leq 1$, we also have 
	\[
	\max t\phi - \min t\phi < P_{\Lambda}(\varphi) - P_{\Lambda^{c}}(\varphi). 
	\]
	and $\varphi + t\phi$ is also a hyperbolic potential.
	
	In particular, since the null potential is hyperbolic, if we have
	\[
	\max \phi - \min \phi < P_{\Lambda}(0)=P(0)=h(f),
	\]
	then $\phi$ is also a hyperbolic potential.
\end{example}

\begin{example}
	Now, for Viana maps, we construct a potential with uniformly bounded Birkhoff sums. 
	
	Let $B$ be an open set and $V=f^{-1}(B)$ such that $V \cap B = \emptyset$ and $V \cap \mathcal{C} = \emptyset$, where $\mathcal{C}$ is the critical set. Let $\phi : \overline{B} \to \mathbb{R}$ be a $C^{\infty}$ function such that $\phi_{\mid \partial B} \equiv 0$ and we define a potential $\varphi : X \to \mathbb{R}$ as
	\[
	\varphi(x) =
	\left\{
	\begin{array}{cc}
		\phi(x), & \text{if} \,\, x  \in \displaystyle B \\
		-\phi(f(x)), &  \text{if} \,\, x \in V\\
		0, &  \text{if} \,\, x \in (V \cup B)^{c}\\    
	\end{array}
	\right.  
	\]
	\begin{claim}
		The Birkhoff sums $S_{n}\varphi$ are uniformly bounded.
	\end{claim}
%
	
	So, the Birkhoff sums are uniformly bounded and Lemma \ref{BIR} guarantees that $\varphi$ is hyperbolic. Moreover, $\varphi$ is H\"older, which means that we have existence and finiteness of equilibrium state.
\end{example}

\begin{remark}
	We observe that this section is also developed for nonexponential contractions, that is, for general zooming systems with the mild condition $\alpha_{n}(r)\leq ar$ for some $a \in (0,1)$. In the case of exponential contractions, we emphasize the relation with the work in \cite{AOS}. The novelty here is the generality of contractions beyond the exponential context.
\end{remark}

\section{Potentials with Uniqueness}\label{PrF}		

Here we prove denseness in Theorem \ref{F}. We divide the proof into some lemmas.

\begin{lemma}\label{dense}
	Let $M$ be a compact space and $f: M \to M$ a continuous open zooming system. There exists a countable set of H\"older potentials which is dense in the set of continuous zooming potentials which have finiteness of equilibrium states.
\end{lemma} 
\begin{proof}
	Once $M$ is compact, by \cite{OV}[Theorem A.3.13], there exists a dense countable set $\mathcal{S} \subset \mathcal{C}^{0}(M)$, the space of continuous potentials. The space of continuous H\"older potentials $\mathcal{H}$ is dense. Given $\phi_{n} \in \mathcal{S}$ and $m \in \mathbb{N}$, there exists $\phi_{n}^{m} \in \mathcal{H}$ such that $\parallel \phi_{n} - \phi_{n}^{m} \parallel < 1/m$, which shows that the countable set $\mathcal{S}_{0} = \{\phi_{n}^{m}\}$ is dense in $\mathcal{H}$.
	
	We remind that the set $\mathcal{ZC}$ of continuous zooming potentials is open. Since the space of H\"older potentials $\mathcal{H}$ is residual, the intersection $\mathcal{ZH} = \mathcal{ZC} \cap \mathcal{S}_{0}$ is dense in $\mathcal{ZC}$. 
	
	
	
\end{proof}

\begin{lemma}\label{sequence}
	Let $M$ be a compact metric space and $f: M \to M$ a continuous open zooming system. Given $\phi : M \to \mathbb{R}$ a continuous zooming potential with finiteness of equilibrium states. For each equilibrium state $\mu$ of $\phi$ there exist a sequence of continuous potentials $\phi_{n} \to \phi$ and $n_{0} \in \mathbb{N}$ such that $\mu$ is the unique measure which is an equilibrium state of $\phi_{n}$ for every $n \geq n_{0}$.
\end{lemma} 
\begin{proof}
	Let $\phi : M \to 
	\mathbb{R}$ be a potential with finiteness of equilibrium states. Let $\mu_{1}, \dots, \mu_{k}$ be its ergodic equilibrium states.  We fix some $i \leq k$ and take a continuous potential $\psi_{i} : M \to \mathbb{R}$  such that $\int\psi_{i} d\mu_{i} < \int \psi_{i} d\mu_{j}, j \neq i$. Define for $a_{n} \downarrow 0$
	\[
	\varphi_{n}^{i}(x) = \phi(x) - a_{n}\bigg{(}\psi_{i}(x) - \int \psi_{i} d\mu_{i}\bigg{)}.
	\]
	Given $\eta \neq \mu_{j}$ for every $j \leq k$ it holds that $\eta$ is not an equilibrium state of $\varphi_{n}^{i}$ for infinitely many $n$.  In fact, once $\eta \neq \mu_{j}$ if we could find $n_{1}, n_{2}, \dots$ such that $\eta$ is an equilibrium state of $\varphi_{n_{1}}^{i}, \varphi_{n_{2}}^{i}, \dots$, by stability we would have $\eta$ as an equilibrium state of $\phi$, which is a contradiction. Then, $\eta$ can only be an equilibrium state of $\varphi_{n}^{i}$ for infinitely many $n$ if $\eta = \mu_{j}$ for some $j \leq k$. If $\eta = \mu_{j}$ for some $j \neq i$, we obtain
	\[
	h_{\eta}(f) + \int \varphi_{n}^{i} d\eta = h_{\eta}(f) + \int \phi d\eta - a_{n}\int \bigg{(}\psi_{i}(x) - \int \psi_{i} d\mu_{i}\bigg{)} d \eta < 
	\]
	\[
	h_{\eta}(f) + \int \phi d\eta = h_{\mu_{i}}(f) + \int \phi d\mu_{i} = h_{\mu_{i}}(f) + \int \varphi_{n}^{i} d\mu_{i}.
	\]
	Then, we can find $n_{0} \in \mathbb{N}$ such that $\mu_{i}$ is the unique measure which is an equilibrium state of $\varphi_{n}^{i}$ for every $n \geq n_{0}$.  
\end{proof}
\begin{lemma}\label{accumulation}
	Let $M$ be a compact metric space and $f: M \to M$ a continuous open zooming system. Given $\phi : M \to \mathbb{R}$ a continuous zooming potential with finiteness of equilibrium states. For each equilibrium state $\mu$ of $\phi$ there exist a sequence of continuous potentials $\phi_{n} \to \phi$ such that $\mu$ is the unique accumulation point of equilibrium states of the sequence $\phi_{n}$.
\end{lemma} 
\begin{proof}
	With the notation of the proof of Lemma \ref{sequence}, let $\eta \neq \mu_{i}$. We claim that there exists a neighbourhood $\mathcal{U}_{\eta}$ of $\eta$ such that for every $\nu \in \mathcal{U}_{\eta}$ we do not have $\nu$ as an equilibrium state of $\varphi_{n}^{i}$ for every $n \geq n_{0}$. In fact, for every $n \geq n_{0}$ we have
	\[
	h_{\eta}(f) + \int \varphi_{n}^{i} d\eta  < h_{\mu_{i}}(f) + \int \varphi_{n}^{i} d\mu_{i},
	\]
	because $\eta$ is not an equilibrium state of $\varphi_{n}^{i}$. Taking any sequence $\eta_{p} \to \eta$, by \cite{BiS}[Lemma 3.0.1] we can find a generating partition $\mathcal{P}$ such that
	\[
	h_{\eta_{p}}(f) \leq \limsup_{p \to \infty}h_{\eta_{p}}(f) \leq h_{\eta}(f,\mathcal{P}) \leq h_{\eta}(f).
	\]
	Hence, we obtain for every $n \geq n_{0}$
	\[
	h_{\eta_{p}}(f) + \int \varphi_{n}^{i} d\eta_{p} \leq \limsup_{p \to \infty}\bigg{(}h_{\eta_{p}}(f) + \int \varphi_{n}^{i} d\eta_{p} \bigg{)} \leq h_{\eta}(f) + \int \varphi_{n}^{i} d\eta  < h_{\mu_{i}}(f) + \int \varphi_{n}^{i} d\mu_{i}.
	\]
	So, $\eta_{p}$ is not an equilbrium state for $\varphi_{n}^{i}$ and the neighborhood can be found because the sequence of measures $\eta_{p} \to \eta$ is arbitrary.
	
	Taking $\eta = \mu_{j}$ for some $j \neq i$, there exist $n_{j}$ and a neighborhood $\mathcal{U}_{\eta}$ of $\eta$ such that $\nu$ is not an equilibrium state of $\varphi_{n}^{i}$ for every $\nu \in \mathcal{U}_{\eta}$ and every $n \geq n_{j}$. If there exists $m_{1}, m_{2}, \dots$ and equilibrium states $\eta_{1}, \eta_{2}, \dots$ of $\varphi_{m_{1}}^{i}, \varphi_{m_{2}}^{i}, \dots$, respectively, and we suppose $\eta_{p} \to \eta$, then
	\[
	h_{\eta_{p}}(f) + \int \varphi_{m_{p}}^{i} d\eta_{p}  = h_{\mu_{i}}(f) + \int \varphi_{m_{p}}^{i} d\mu_{i} = h_{\mu_{i}}(f) + \int \phi d\mu_{i} = h_{\eta}(f) + \int \phi d\eta > h_{\eta}(f) + \int \varphi_{m_{p}}^{i} d\eta.
	\]
	So, we have $\eta_{p} \not \in \mathcal{U}_{\eta}$ with $p$ large enough. It is a contradiction. It means that, in this case, we must have $\eta = \mu_{i}$. Also, $\mu_{i}$ is the unique accumulation point of equilibrium states of the sequence $\varphi_{n}^{i} $.
\end{proof}
\begin{lemma}\label{unique}
	Let $M$ be a compact metric space and $f: M \to M$ a continuous open zooming system. The set of continuous potentials with uniqueness is dense in the set of continuous potentials with finiteness.
\end{lemma} 
\begin{proof}
	With the notation of Lemma \ref{accumulation}, given $\mu_{i}$ an equilibrium state, we have that it is the unique accumulation point of the sequence $\varphi_{n}^{i}$ defined in Lemma \ref{sequence}. If we have a non constant sequence $\eta_{p} \to \mu_{i}$, where $\eta_{p}$ is an equilibrium state of $\varphi_{n_{p}}^{i} \to \phi$, then for every $p \in \mathbb{N}$ it holds that
	\[
	h_{\eta_{p}}(f) + \int \varphi_{n_{p}}^{i} d\eta_{p} = h_{\mu_{i}}(f) + \int \varphi_{n_{p}}^{i} d\mu_{i} = h_{\mu_{i}}(f) + \int \phi d\mu_{i}
	\]
	Also, there exists $n_{p}' \in \mathbb{N}$ such that for every $n \geq n_{p}'$
	\[
	h_{\eta_{p}}(f) + \int \varphi_{n}^{i} d\eta_{p} < h_{\mu_{i}}(f) + \int \varphi_{n}^{i} d\mu_{i} = h_{\mu_{i}}(f) + \int \phi d\mu_{i}
	\]
	But for every $p \in \mathbb{N}$ and every $n \in \mathbb{N}$
	\[
	h_{\eta_{p}}(f) + \int \varphi_{n}^{i} d\eta_{p} = h_{\eta_{p}}(f) + \int \phi d\eta_{p} - a_{n}\bigg{(}\int \psi_{i}d\eta_{p} - \int \psi_{i} d\mu_{i}\bigg{)}
	\]
	and, in particular, for $n = n_{p}$
	\[
	h_{\eta_{p}}(f) + \int \varphi_{n_{p}}^{i} d\eta_{p} = h_{\eta_{p}}(f) + \int \phi d\eta_{p} - a_{n_{p}}\bigg{(}\int \psi_{i}d\eta_{p} - \int \psi_{i} d\mu_{i}\bigg{)} = h_{\mu_{i}}(f) + \int \phi d\mu_{i}
	\]
	Moreover, once $\mu_{i}$ is an equilibrium state of $\phi$, we have
	\[
	h_{\eta_{p}}(f) + \int \phi d\eta_{p} \leq h_{\mu_{i}}(f) + \int \phi d\mu_{i} \implies - a_{n_{p}}\bigg{(}\int \psi_{i}d\eta_{p} - \int \psi_{i} d\mu_{i}\bigg{)} \geq 0 \implies 
	\]
	\[
	-a_{n}\bigg{(}\int \psi_{i}d\eta_{p} - \int \psi_{i} d\mu_{i}\bigg{)} \geq 0
	\]
	We claim that $n_{p} = 1$. Otherwise, since the sequence $a_{n}$ is decreasing, we have $a_{n_{p} -1} > a_{n{p}}$ and
	\[
	h_{\mu_{i}}(f) + \int \phi d\mu_{i} \geq h_{\eta_{p}}(f) + \int \varphi_{n_{p}-1}^{i} d\eta_{p} = h_{\eta_{p}}(f) + \int \phi d\eta_{p} - a_{n_{p}-1}\bigg{(}\int \psi_{i}d\eta_{p} - \int \psi_{i} d\mu_{i}\bigg{)} > 
	\]
	\[
	h_{\eta_{p}}(f) + \int \phi d\eta_{p} - a_{n_{p}}\bigg{(}\int \psi_{i}d\eta_{p} - \int \psi_{i} d\mu_{i}\bigg{)} = h_{\eta_{p}}(f) + \int \varphi_{n_{p}}^{i} d\eta_{p} = h_{\mu_{i}}(f) + \int \phi d\mu_{i},
	\]
	which is a contradiction. Hence, $n_{p} =1$ for every $p \in \mathbb{N}$ and for $n > 1$ the unique equilibrium state of $\varphi_{n}^{i}$ is $\mu_{i}$.
\end{proof}
\begin{lemma}\label{countable}
	Let $M$ be a compact metric space and $f: M \to M$ a continuous open zooming system. There exists a countable set of continuous potentials with uniqueness which is dense in the set of continuous potentials with finiteness.
\end{lemma} 
\begin{proof}
	By Lemma \ref{dense}, there exists a countable set $\mathcal{ZH} = \{\phi_{m}\}$ of H\"older zooming potentials which is dense in the set of continuous potentials with finiteness. For each $m \in \mathbb{N}$ we can define the following sequence as in Lemma \ref{sequence}:
	\[
	\varphi_{m,n}^{i_{m}}(x) = \phi_{m}(x) - a_{n}\bigg{(}\psi_{m,i_{m}}(x) - \int \psi_{m,i_{m}} d\mu_{i_{m}}\bigg{)}, i_{m} \leq k_{m}.
	\]
	By Lemmas \ref{accumulation} and \ref{unique} this sequence of potentials has uniqueness. Hence, the following set of potentials has uniqueness and is dense $\{\varphi_{m,n}^{i_{m}}\}, m,n \in \mathbb{N}, i_{m} \leq k_{m}$.	
\end{proof}
\begin{lemma}\label{convex}
	Let $M$ be a compact metric space and $f: M \to M$ a continuous open zooming system. For each $\mu$ ergodic probability which is an equilibrium state for some potential, we have that the following set is closed and convex (in particular, connected):
	\[
	\mathcal{E}_{\mu} = \{\phi: M \to \mathbb{R} \mid \mu \,\, \text{is an equilibrium state of} \,\, \phi\}.
	\]
	Also, the following set is convex (and connected):
	\[
	\mathcal{E}_{\mu}' = \{\phi: M \to \mathbb{R} \mid \mu \,\, \text{is the unique equilibrium state of} \,\, \phi\}.
	\]
\end{lemma} 
\begin{proof}
	Let $\phi, \varphi \in \mathcal{E}_{\mu}$ and $t \in (0,1)$. It holds that
	\[
	P((1-t)\phi + t\varphi) = \sup_{\eta}\Bigg{\{} h_{\eta}(f) + \int[(1-t)\phi + t\varphi] d\eta\Bigg{\}} =
	\]
	\[
	\sup_{\eta}\Bigg{\{}(1-t) \Bigg{(}h_{\eta}(f) + \int \phi d\eta\Bigg{)}+ t\Bigg{(}h_{\eta}(f) + \int \varphi d\eta \Bigg{)}\Bigg{\}} \leq
	\]
	\[
	\sup_{\eta}\Bigg{\{}(1-t) \Bigg{(}h_{\eta}(f) + \int \phi d\eta\Bigg{)}\Bigg{\}} + 
	\sup_{\eta}\Bigg{\{} t\Bigg{(}h_{\eta}(f) + \int \varphi d\eta \Bigg{)}\Bigg{\}} =
	\]
	\[
	(1-t)\sup_{\eta}\Bigg{\{}h_{\eta}(f) + \int \phi d\eta\Bigg{\}} + 
	t\sup_{\eta}\Bigg{\{}h_{\eta}(f) + \int \varphi d\eta\Bigg{\}} =
	\]
	\[
	(1-t)P(\phi) + tP(\varphi) = (1-t)\Bigg{(}h_{\mu}(f) + \int \phi d\mu \Bigg{)} + t\Bigg{(}h_{\mu}(f) + \int \varphi d\mu \Bigg{)}=
	\]
	\[
	h_{\mu}(f) + \int [(1-t)\phi + t\varphi] d\mu.
	\]
	It implies that $P((1-t)\phi + t\varphi) = h_{\mu}(f) + \int [(1-t)\phi + t\varphi] d\mu$ and $\mu$ is an equilibrium state of $(1-t)\phi + t\varphi$ for every $t \in (0,1)$. Moreover, if $\varphi$ has uniqueness, given $\nu \neq \mu$, we have that $\nu$ is not an equlibrium state of $\varphi$. Hence,
	\[
	h_{\nu}(f) + \int [(1-t)\phi + t\varphi] d\nu = (1-t)\Bigg{(}h_{\nu}(f) + \int \phi d\nu \Bigg{)} + t\Bigg{(}h_{\nu}(f) + \int \varphi d\nu \Bigg{)} <
	\]
	\[
	(1-t)P(\phi) + tP(\varphi) = h_{\mu}(f) + \int [(1-t)\phi + t\varphi] d\mu = P((1-t)\phi + t\varphi),
	\] 
	It means that $\nu$ cannot be an equilibrium state of $(1-t)\phi + t\varphi$ and $\mu$ is the unique equilibrium state of $(1-t)\phi + t\varphi$ for every $t \in (0,1)$. It shows that the sets $\mathcal{E}_{\mu}, \mathcal{E}_{\mu}'$ are convex (and connected). By stability, it is easy to see that the set $\mathcal{E}_{\mu}$ is closed. The lemma is proved.
\end{proof} 
\begin{lemma}\label{open}
	Let $M$ be a compact metric space and $f: M \to M$ a continuous open zooming system. The set of continuous potentials with uniqueness contains an open and dense set in the set of continuous potentials with finiteness.
\end{lemma} 
\begin{proof}
	Let $\phi: M \to \mathbb{R}$ a potential with at least two equilibrium states $\mu_{1}, \dots, \mu_{k}$. We can construct a sequence as in Lemma \ref{sequence} with unique equilibrium state $\mu_{i}$:
	\[
	\varphi_{n}^{i}(x) = \phi(x) - a_{n}\bigg{(}\psi_{i}(x) - \int \psi_{i} d\mu_{i}\bigg{)},
	\] 
	with $\psi_{i}$ chosen such that $\int\psi_{i} d\mu_{i} < \int \psi_{i} d\mu_{j}, j \neq i$ and $a_{n} \downarrow 0$. We can take a neighborhood $\mathcal{U}_{i}$ of $\psi_{i}$ such that for every $\psi \in \mathcal{U}_{i}$ we have $\int\psi d\mu_{i} < \int \psi d\mu_{j}, j \neq i$. Then, we can obtain the following neighborhood of $\varphi_{n}^{i}(x)$
	\[
	\bigg{\{}\phi(\cdot) - a_{n}\bigg{(}\psi(\cdot) - \int \psi d\mu_{i}\bigg{)} \bigg{|} \psi \in \mathcal{U}_{i}{\bigg\}} \subset \mathcal{E}_{\mu_{i}}',
	\]
	where the set $\mathcal{E}_{\mu_{i}}'$ is defined in Lemma \ref{convex}.
	
	If the countable subset $\mathcal{CF} \subset \mathcal{ZH}$ of zooming H\"older potentials with at least two ergodic equilibrium states is dense in the set of continuous potentials with finiteness, we can take the open and dense set
	\[
	\mathcal{D}:=\bigcup_{m=1}^{\infty}\bigcup_{i_{m}=1}^{k_{m}}\bigg{\{}\phi_{m}(\cdot) - a_{n}\bigg{(}\psi(\cdot) - \int \psi d\mu_{i_{m}}\bigg{)} \bigg{|} \psi \in \mathcal{U}_{i_{m}}{\bigg\}},
	\]
	where $\phi_{m} \in \mathcal{CF}$ given in Lemma \ref{countable}. Otherwise, there exists a neighbourhood $\mathcal{U} \subset \mathcal{ZC} \backslash \mathcal{CF}$, the interior of the complement of $\mathcal{CF}$,  composed of potentials with uniqueness. We then take $\mathcal{D} \cup \mathcal{U}$. 
\end{proof}

\section{Examples}\label{Examples}

In this section, we give examples of  zooming systems. We begin by defining a non-flat map. We begin by recalling the examples given in \cite{AOS}, where the expanding set is dense in $M$, the hole is empty and the map is closed.

\subsection{Viana maps} We recall the definition of the open class of maps with critical sets in dimension 2, introduced by M. Viana in \cite{V}. We skip the technical
points. It can be generalized for any dimension (See \cite{A}).

Let $a_{0} \in (1,2)$ be such that the critical point $x=0$ is pre-periodic for the quadratic map $Q(x)=a_{0} - x^{2}$. Let $S^{1}=\mathbb{R}/\mathbb{Z}$ and 
$b:S^{1} \to \mathbb{R}$ a Morse function, for instance $b(\theta) = \sin(2\pi\theta)$. For fixed small $\alpha > 0$, consider the map
\[
\begin{array}{c}
	f_{0}: S^{1} \times \mathbb{R} \longrightarrow S^{1} \times \mathbb{R}\\
	\,\,\,\,\,\,\,\,\,\,\,\,\,\,\,\,\,\,\,\ (\theta,x) \longmapsto (g(\theta),q(\theta,x))
\end{array}
\] 
where $g$ is the uniformly expanding map of the circle defined by $g(\theta)=d\theta
(mod\mathbb{Z})$ for some $d \geq 16$, and $q(\theta,x) = a(\theta) - x^{2}$ with $a(\theta) = a_{0} + \alpha b(\theta)$. It is easy to check that for $\alpha > 0$ 
small enough there is an interval $I \subset (-2,2)$ for which $f_{0}(S^{1} \times I)$ is contained in the interior of $S^{1} \times I$. Thus, any map $f$ sufficiently
close to $f_{0}$ in the $C^{0}$ topology has $S^{1} \times I$ as a forward invariant region. We consider from here on these maps $f$ close to $f_{0}$ restricted to 
$S^{1} \times I$. Taking into account the expression of $f_{0}$ it is not difficult to check that for $f_{0}$ (and any map $f$ close to $f_{0}$ in the $C^{2}$ topology)
the critical set is non-degenerate.

The main properties of $f$ in a $C^{3}$ neighbourhood of $f$ that we will use here are summarized below (See \cite{A},\cite{AV},\cite{Pi1}):

\begin{enumerate}
	\item[(1)] $f$ is \textbf{\textit{non-uniformly expanding}}, that is, there exist $\lambda > 0$ and a Lebesgue full measure set $H \subset S^{1} \times I$ such that 
	for every point $p=(\theta, x) \in H$, the following holds
	\[
	\displaystyle \limsup_{n \to \infty} \frac{1}{n} \sum_{i=0}^{n-1} \log \parallel Df(f^{i}(p))^{-1}\parallel^{-1} < -\lambda.
	\]  
	\item[(2)] Its orbits have \textbf{\textit{slow approximation to the critical set}}, that is, for every $\epsilon > 0$ the exists $\delta > 0$ such that for every point
	$p=(\theta, x) \in H \subset S^{1} \times I$, the following holds 
	\[
	\displaystyle \limsup_{n \to \infty} \frac{1}{n} \sum_{i=0}^{n-1} - \log \text{dist}_{\delta}(p,\mathcal{C}) < \epsilon.
	\]  
	where 
	\[
	\text{dist}_{\delta}(p,\mathcal{C}) =  
	\left\{ 
	\begin{array}{ccc}
		dist(p,\mathcal{C}), & if & dist(p,\mathcal{C}) < \delta\\
		1 & if & dist(p,\mathcal{C}) \geq \delta 
	\end{array}
	\right.
	\]  
	\item[(3)] $f$ is topologically mixing;
	
	\item[(4)] $f$ is strongly topologically transitive;
	
	\item[(5)] it has a unique ergodic absolutely continuous invariant (thus SRB) measure;
	
	\item[(6)]the density of the SRB measure varies continuously in the $L^{1}$ norm with $f$.
\end{enumerate}

\begin{remark}
	We observe that this definition of non-uniformly expansion is included in ours by neighbourhoods.
\end{remark}

\subsection{Benedicks-Carleson Maps} We study a class of non-hyperbolic maps of the interval with the condition of exponential growth of the derivative at critical values, called 
\textbf{\textit{Collet-Eckmann Condition}}. We also ask the map to be $C^{2}$ and topologically mixing and the critical points to have critical order 
$2 \leq \alpha < \infty$.

Given a critical point $c \in I$, the \textbf{\textit{critical order}} of $c$ is a number $\alpha_{c} > 0$ such that 
$f(x) = f(c) \pm |g_{c}(x)| ^{\alpha_{c}}, \,\, \text{for all} \, \, x \in \mathcal{U}_{c}$ where $g_{c}$ is a diffeomorphism 
$g_{c}: \mathcal{U}_{c} \to g(\mathcal{U}_{c})$ and $\mathcal{U}_{c}$ is a neighbourhood of $c$. 

Let $\delta>0$ and denote $\mathcal{C}$ the set of critical points and $\displaystyle B_{\delta} = \cup_{c \in \mathcal{C}} (c - \delta, c + \delta)$. 
Given $x \in I$, we suppose that

\begin{itemize}
	\item \textbf{(Expansion outside $B_{\delta}$)}.  There exists $\kappa > 1 $ and $\beta > 0$ such that, if $x_{k} = f^{k}(x) \not \in B_{\delta}, \,\, 0 \leq k \leq n-1$ then $|Df^{n}(x)| \geq \kappa \delta^{(\alpha_{\max} -1)}e^{\beta n}$, where $\alpha_{\max} = \max \{\alpha_{c}, c \in \mathcal{C}\}$. Moreover, if $x_{0} \in f(B_{\delta})$ or $x_{n} \in B_{\delta}$ then $|Df^{n}(x)| \geq \kappa e^{\beta n}$.
	
	\item \textbf{(Collet-Eckmann Condition)}. There exists $\lambda > 0$ such that 
	\[
	|Df^{n}(f(c))| \geq e^{\lambda n}.
	\]
	
	\item \textbf{(Slow Recurrence to $\mathcal{C}$)}. There exists $\sigma \in (0, \lambda/5)$ such that 
	\[
	dist(f^{k}(x), \mathcal{C}) \geq e^{-\sigma k}.
	\]
\end{itemize}

\subsection{Rovella Maps}

There is a class of non-uniformly expanding maps known as \textbf{\textit{Rovella Maps}}. They are derived from the so-called \textit{Rovella Attractor},
a variation of the \textit{Lorenz Attractor}. We proceed with a brief presentation. See \cite{AS} for details.

\subsubsection{Contracting Lorenz Attractor}

The geometric Lorenz attractor is the first example of a robust attractor for a flow containing a hyperbolic singularity. The attractor is a transitive maximal invariant
set for a flow in three-dimensional space induced by a vector field having a singularity at the origin for which the derivative of the vector field at the singularity has
real eigenvalues $\lambda_{2} < \lambda_{3} < 0 < \lambda_{1}$ with $\lambda_{1} + \lambda_{3} > 0$. The singularity is accumulated by regular orbits which prevent the 
attractor from being hyperbolic.

The geometric construction of the contracting Lorenz attractor (Rovella attractor) is the same as the geometric Lorenz attractor. The only difference is the condition
(A1)(i) below that gives in particular $\lambda_{1} + \lambda_{3} < 0$. The initial smooth vector field $X_{0}$ in $\mathbb{R}^{3}$ has the following properties:

\begin{itemize}
	
	\item[(A1)] $X_{0}$ has a singularity at $0$ for which the eigenvalues $\lambda_{1},\lambda_{2},\lambda_{3} \in \mathbb{R}$ of $DX_{0}(0)$ satisfy:
	\begin{itemize}
		
		\item[(i)] $0 < \lambda_{1} < -\lambda_{3}  < -\lambda_{2}$,
		
		\item[(ii)] $r > s+3$, where $r=-\lambda_{2}/\lambda_{1}, s=-\lambda_{3}/\lambda_{1}$;
	\end{itemize}
	
	\item[(A2)] there is an open set $U \subset \mathbb{R}^{3}$, which is forward invariant under the flow, containing the cube
	$\{(x,y,z) : \mid x \mid \leq 1, \mid y \mid \leq 1, \mid x \mid \leq 1\}$ and supporting the \textit{Rovella attractor}
	\[
	\displaystyle \Lambda_{0} = \bigcap_{t \geq 0} X_{0}^{t}(U).
	\]
	
	The top of the cube is a Poincar\'e section foliated by stable lines $\{x = \text{const}\} \cap \Sigma$ which are invariant under Poincar\'e first return map $P_{0}$.
	The invariance of this foliation uniquely defines a one-dimensional map $f_{0} : I \backslash \{0\} \to I$ for which
	\[
	f_{0} \circ \pi = \pi \circ P_{0},
	\]
	where $I$ is the interval $[-1,1]$ and $\pi$ is the canonical projection $(x,y,z) \mapsto x$;
	
	\item[(A3)] there is a small number $\rho >0$ such that the contraction along the invariant foliation of lines $x =$const in $U$ is stronger than $\rho$.
\end{itemize}

See \cite{AS} for properties of the map $f_{0}$.

\subsubsection{Rovella Parameters}

The Rovella attractor is not robust. However, the chaotic attractor persists in a measure theoretical sense: there exists a one-parameter family of positive Lebesgue measure
of $C^{3}$ close vector fields to $X_{0}$ which have a transitive non-hyperbolic attractor. In the proof of that result, Rovella showed that there is a set of parameters
$E \subset (0,a_{0})$ (that we call \textit{Rovella parameters}) with $a_{0}$ close to $0$ and $0$ a full density point of $E$, i.e.
\[
\displaystyle \lim_{a \to 0} \frac{\mid E \cap (0,a) \mid}{a} = 1,
\]
such that:

\begin{itemize}
	\item[(C1)] there is $K_{1}, K_{2} > 0$ such that for all $a \in E$ and $x \in I$
	\[
	K_{2} \mid x \mid^{s-1} \leq f_{a}'(x) \leq K_{1} \mid x \mid^{s-1},
	\]
	where $s=s(a)$. To simplify, we shall assume $s$ fixed.
	
	\item[(C2)] there is $\lambda_{c} > 1$ such that for all $a \in E$, the points $1$ and $-1$ have \textit{Lyapunov exponents} greater than $\lambda_{c}$:
	\[
	(f_{a}^{n})'(\pm 1) > \lambda_{c}^{n}, \,\, \text{for all} \, \, n \geq 0;
	\]
	
	\item[(C3)] there is $\alpha > 0$ such that for all $a \in E$ the \textit{basic assumption} holds:
	\[
	\mid f_{a}^{n-1}(\pm 1)\mid > e^{-\alpha n}, \,\, \text{for all} \, \, n \geq 1;
	\]
	
	\item[(C4)] the forward orbits of the points $\pm 1$ under $f_{a}$ are dense in $[-1,1]$ for all $a \in E$.
\end{itemize}

\begin{definition}
	We say that a map $f_{a}$ with $a \in E$ is a \textbf{\textit{Rovella Map}}. 
\end{definition}

\begin{theorem}
	(Alves-Soufi \cite{AS}) Every Rovella map is non-uniformly expanding. 
\end{theorem}

\subsection{Hyperbolic Times}

The idea of hyperbolic times is a key notion on the study of non-uniformly hyperbolic dynamics and it was introduced by Alves et al. 
This is powerful to get expansion in the context of non-uniform expansion. Here, we recall the basic definitions and results on hyperbolic times that we will use later on. 
We will see that this notion is an example of a Zooming Time. 

In the following, we give definitions taken from \cite{A} and \cite{Pi1}.

\begin{definition}
	Let $M$ be a compact Riemannian manifold of dimension $d \geq 1$ and $f:M \to M$ a continuous map defined on $M$.
	The map $f$ is called \textbf{\textit{non-flat}} if it is a local $C^{1 + \alpha}, (\alpha >0)$ diffeomorphism in the whole manifold except in a 
	non-degenerate set $\mathcal{C} \subset M$. We say that $M \neq \mathcal{C} \subset M$ is a \textbf{\textit{non-degenerate set}}
	if there exist $\beta, B > 0$ such that the following two conditions hold.
	
	\begin{itemize}
		\item $\frac{1}{B} d(x,\mathcal{C})^{\beta} \leq \frac{\parallel Df(x) v\parallel}{\parallel v \parallel} \leq B d(x,\mathcal{C})^{-\beta}$ for all $v \in T_{x}M$, for every $x \in M\backslash\mathcal{C}$.
		
		For every $x, y \in M \backslash \mathcal{C}$ with $d(x,y) < d(x,\mathcal{C})/2$ we have
		
		\item $\mid \log \parallel Df(x)^{-1} \parallel - \log \parallel Df(y)^{-1} \parallel \mid \leq \frac{B}{d(x,\mathcal{C})^{\beta}} d(x,y)$.
	\end{itemize}
	
\end{definition}

In the following, we give the definition of a hyperbolic time \cite{ALP2}, \cite{Pi1}.

\begin{definition}
	(Hyperbolic times). Let us fix $0 < b = \frac{1}{3} \min\{1,1 \slash \beta\} < \frac{1}{2} \min\{1,1\slash \beta\}$. 
	Given $0 < \sigma < 1$ and $\epsilon > 0$, we will say that $n$ is a $(\sigma, \epsilon)$\textbf{\textit{-hyperbolic time}} for a point $x \in M$ 
	(with respect to the non-flat map $f$ with a $\beta$-non-degenerate critical/singular set $\mathcal{C})$ if for all $1 \leq k \leq n$ we have 
	\[
	\prod_{j=n-k}^{n-1} \|(Df \circ f^{j}(x))^{-1}\| \leq \sigma^{k} \,\, \text{and} \,\, dist_{\epsilon}(f^{n-k}(x), \mathcal{C}) \geq \sigma^{bk}.
	\]
	where
	\[
	\text{dist}_{\epsilon}(p,\mathcal{C}) =  
	\left\{ 
	\begin{array}{ccc}
		dist(p,\mathcal{C}), & if & dist(p,\mathcal{C}) < \epsilon\\
		1 & if & dist(p,\mathcal{C}) \geq \epsilon. 
	\end{array}
	\right.
	\]
	We denote de set of points of $M$ such that $n \in \mathbb{N}$ is a $(\sigma,\epsilon)$-hyperbolic time by $H_{n}(\sigma,\epsilon,f)$.
\end{definition}

\begin{proposition}
	(Positive frequence). Given $\lambda > 0$ there exist $\theta > 0$ and $\epsilon_{0} > 0$ such that, for every $x \in M$ and $\epsilon \in (0,\epsilon_{0}]$,
	\[
	\#\{1 \leq j \leq n \mid \,\, x \in H_{j}(e^{-\lambda \slash 4}, \epsilon, f) \} \geq \theta n,
	\]
	whenever $\frac{1}{n}\sum_{i=0}^{n-1}\log\|(Df(f^{i}(x)))^{-1}\|^{-1} \geq \lambda$ and $\frac{1}{n}\sum_{i=0}^{n-1}-\log dist_{\epsilon}(x, \mathcal{C}) \leq \frac{\lambda}{16 \beta}$.
\end{proposition}

Denote by $\mathcal{H}$ the set of point $x \in M$ such that
\[
\displaystyle \limsup_{n \to \infty} \frac{1}{n} \sum_{i=0}^{n-1} \log \parallel Df(f^{i}(p))^{-1}\parallel^{-1} < -\lambda.
\] 
and
\[
\displaystyle \limsup_{n \to \infty} \frac{1}{n} \sum_{i=0}^{n-1} - \log \text{dist}_{\delta}(p,\mathcal{C}) < \epsilon.
\]
If $f$ is non-uniformly expanding, it follows from the proposition that the points of $\mathcal{H}$ have infinitely many moments with positive frequency of hyperbolic times. In particular, they have infinitely many hyperbolic times.

The following proposition shows that the hyperbolic times are indeed zooming times, where the zooming contraction is $\alpha_{k}(r) = \sigma^{k/2}r$.

\begin{proposition}
	Given $\sigma \in (0,1)$ and $\epsilon > 0$, there is $\delta,\rho > 0$, depending only on $\sigma$ and $\epsilon$ and on the map $f$, such that if $x \in H_{n}(\sigma,\epsilon,f)$ then there exists a neighbourhood $V_{n}(x)$ of $x$ with the following properties:
	
	\begin{enumerate}
		\item[(1)] $f^{n}$ maps $\overline{V_{n}(x)}$ diffeomorphically onto the ball $\overline{B_{\delta}(f^{n}(x))}$;
		\item[(2)] $dist(f^{n-j}(y),f^{n-j}(z)) \leq \sigma^{j\slash 2} dist(f^{n}(y), f^{n}(z)), \text{for all} \, \, y,z \in V_ {n}(x)$ and $1 \leq j < n$.
		\item[(3)]$\log \frac{\mid \det Df^{n}(y)\mid}{\mid \det Df^{n}(z)\mid} \leq \rho d(f^{n}(y),f^{n}(z))$.
	\end{enumerate}
	
	for all $y,z \in V_{n}(x)$.
\end{proposition}

The sets $V_{n}(x)$ are called hyperbolic pre-balls and their images $f^{n}(V_{n}(x)) = B_{\delta}(f^{n}(x))$, hyperbolic balls.

\bigskip

In the following, we give definitions for a map on a metric space to have similar behaviour to maps with hyperbolic times and which can be found in \cite{Pi1}.  

Given $M$ a metric spaces and $f: M \to M$, we define for $p \in M$:
\[
\displaystyle \mathbb{D}^{-}(p) = \liminf_{x \to p} \frac{d(f(x),f(p)}{d(x,p)}
\]
Define also,
\[
\displaystyle \mathbb{D}^{+}(p) = \limsup_{x \to p} \frac{d(f(x),f(p)}{d(x,p)}
\]
We will consider points $x \in M$ such that 
\[
\displaystyle \limsup_{n \to \infty} \frac{1}{n} \sum_{i=0}^{n-1} \log \mathbb{D}^{-} \circ f^{i}(x) > 0.
\]  
The critical set $\mathcal{C}$ is the set of points $x \in M$ such that $\mathbb{D}^{-}(x) = 0$ or $\mathbb{D}^{+}(x) = \infty$. For  the non-degenerateness we ask that $\mathcal{C} \neq M$ and there exist $B, \beta >0$ such that

\begin{itemize}
	
	\item $\frac{1}{B} d(x,\mathcal{C})^{\beta} \leq \mathbb{D}^{-}(x) \leq \mathbb{D}^{+}(x) \leq B d(x,\mathcal{C})^{-\beta}, x \not \in \mathcal{C}$.
	
	For every $x, y \in M \backslash \mathcal{C}$ with $d(x,y) < d(x,\mathcal{C})/2$ we have
	
	\item $\mid \log \mathbb{D}^{-}(x) - \log \mathbb{D}^{-}(y) \mid \leq \frac{B}{d(x,\mathcal{C})^{\beta}} d(x,y)$.
	
\end{itemize}

With these conditions we can see that all the consequences for hyperbolic times are valid here and the expanding sets and measures are zooming sets and measures.

\begin{definition}
	We say that a map is \emph{conformal at p} if $\mathbb{D}^{-}(p) = \mathbb{D}^{+}(p)$. So, we define
	\[
	\displaystyle \mathbb{D}(p) = \lim_{x \to p} \frac{d(f(x),f(p)}{d(x,p)}.
	\]
\end{definition}

Now, we give an example of such an open non-uniformly expanding map.  

\subsection{Expanding sets on a metric space} Let $\sigma : \Sigma_{2}^{+} \to \Sigma_{2}^{+}$ be the one-sided shift, with the usual metric:
\[
\displaystyle d(x,y) = \sum_{n=1}^{\infty} \frac{\mid x_{n} - y_{n} \mid}{2^{n}},
\]
where $x = \{x_{n}\}, y = \{y_{n}\}$. We have that $\sigma$ is a conformal map such that $\mathbb{D}^{-}(x) = 2, \text{for all} \, \, \, x \in \Sigma_{2}^{+}$. Also, every forward invariant set (in particular the whole $\Sigma_{2}^{+}$)  and all invariant measures for the shift $\sigma$ are expanding (then they are zooming). In particular, if we consider an invariant set that is not dense such that the reference measure has a Jacobian with bounded distortion, we can obtain an open shift map with $H \neq \emptyset$. To be precise, by taking any (previously fixed) zooming set $\Lambda \subset \Sigma_{2}^{+}$ which is not dense such that the reference measure has a Jacobian with bounded distortion, we apply \cite{S}[Theorem A] to obtain an open zooming system and a Markov structure adapted to a hole $H \subset \Sigma_{2}^{+}$ such that $H \cap \Lambda = \emptyset$. It is enough to take $r_{0}>0$ as in Theorem \cite{S}[Theorem A] such that one of the balls of the open cover is disjoint from $\Lambda$. Hence, we can apply our Theorems \ref{B} and \ref{C} to obtain equilibrium states.

\subsection{Zooming sets on a metric space (not expanding)} Let $\sigma : \Sigma_{2}^{+} \to \Sigma_{2}^{+}$ be the one-sided shift, with the following metric for $\sum_{n=1}^{\infty} b_{n} < \infty$:
\[
\displaystyle d(x,y) = \sum_{n=1}^{\infty} b_{n}\mid x_{n} - y_{n} \mid,
\]
where $x = \{x_{n}\}, y = \{y_{n}\}$ and $b_{n+k} \leq b_{n}b_{k}$ for all $n,k \geq 1$. By induction, it means that $b_{n} \leq b_{1}^{n}$. Let us suppose that $b_{n} \leq a_{n}:=(n+b)^{-a}, a>1, b>0$ for all $n \geq 1$. 

We claim that $a_{n}$ defines a Lipschitz contraction for the shift map. We require that there exists $n_{0} > 1$ such that $b_{n} > a_{1}^{n} \geq b_{1}^{n}$ for $n \leq n_{0}$. So, the contraction is not exponential. In fact, if $x,y$ belongs to the cylinder $C_{k}$ we have
\begin{eqnarray*}
	\displaystyle d(x,y) &=& \sum_{n=1}^{\infty} b_{n}\mid x_{n} - y_{n} \mid = \sum_{n=k+1}^{\infty} b_{n}\mid x_{n} - y_{n} \mid = \sum_{n=1}^{\infty} b_{n+k}\mid x_{n+k} - y_{n+k} \mid\\
	&\leq& b_{k} \sum_{n=1}^{\infty} b_{n}\mid x_{n+k} - y_{n+k} \mid = b_{k} d(\sigma^{k}(x),\sigma^{k}(y)) \leq a_{k} d(\sigma^{k}(x),\sigma^{k}(y)).
\end{eqnarray*}
It implies that
\begin{eqnarray*}
	\displaystyle d(\sigma^{i}(x),\sigma^{i}(y)) \leq  a_{k-i} d(\sigma^{k-i}(\sigma^{i}(x)),\sigma^{k-i}(\sigma^{i}(y)))= a_{k-i}d(\sigma^{k}(x),\sigma^{k}(y)), i \leq k.
\end{eqnarray*}
It means that the sequence $a_{n}$ defines a Lipschitz contraction, as we claimed.

Now, every forward invariant set (in particular the whole $\Sigma_{2}^{+}$)  and all invariant measures for the shift $\sigma$ are not expanding but they are  zooming. In particular, if we consider an invariant set that is not dense such that the reference measure has a Jacobian with bounded distortion, we can obtain an open shift map with $H \neq \emptyset$. To be precise, by taking any (previously fixed) zooming set $\Lambda \subset \Sigma_{2}^{+}$ which is not dense such that the reference measure has a Jacobian with bounded distortion, we apply \cite{S}[Theorem A] to obtain an open zooming system and a Markov structure adapted to a hole $H \subset \Sigma_{2}^{+}$ such that $H \cap \Lambda = \emptyset$. It is enough to take $r_{0}>0$ as in Theorem \cite{S}[Theorem A] such that one of the balls of the open cover is disjoint from $\Lambda$. Hence, we can apply our Theorems \ref{B} and \ref{C} to obtain equilibrium states.

\subsection{Uniformly expanding maps} As can be seen in \cite{OV} Chapter 11, we have the so-called \textbf{\textit{uniformly expanding maps}} which is defined on a compact differentiable manifold $M$ as a $C^{1}$ map $f:M \to M$ (with no critical set) for which there exists $\sigma > 1$ such that
\[
\|Df(x)v\|\geq \sigma \|v\|, \,\, \text{for every} \,\, x \in M, v \in T_{x}M.
\]

For compact metric spaces $(M,d)$ we define it as a continuous map $f:M \to M$, for which there exists $\sigma > 1, \delta>0$ such that for every $x \in M$ we have that the image of the ball $B(x,\delta)$ contains a neighbourhood of the ball $B(f(x),\delta)$ and
\[
d(f(a),f(b)) \geq \sigma d(a,b), \,\, \text{for every} \,\, a,b \in B(x,\delta).
\]
We observe that the uniformly expanding maps on differentiable manifolds satisfy the conditions for the definition on compact metric spaces, when they are seen as Riemannian manifolds.

\subsection{Local diffeomorphisms}\label{local} As can be seen in details in \cite{A}, we will briefly describe a class of non-uniformly expanding maps.

Here we present a robust ($C^{1}$ open) classes of local diffeomorphisms (with no critical set) that are non-uniformly expanding. Such classes of maps can be obtained, e.g., through deformation of a uniformly expanding map by isotopy inside some small region. In general, these maps are not uniformly expanding: deformation can be made in such way that the new map has periodic saddles.

Let $M$ be a compact manifold supporting some uniformly expanding map $f_{0}$. $M$ could be the $d$-dimensional torus $\mathbb{T}^{d}$, for instance. Let $V \subset M$ be some small compact domain, so that the restriction of $f_{0}$ to $V$ is injective. Let $f$ be any map in a sufficiently small $C^{1}$-neighbourhood $\mathcal{N}$ of $f_{0}$ so that:

\begin{itemize}
	\item $f$ is \textit{volume expanding everywhere}: there exists $\sigma_{1} > 1$ such that
	\[
	|\det Df(x)| > \sigma_{1} \,\, \text{for every} \,\, x \in M;
	\]	
	
	\item $f$ is \textit{expanding outside} $V$: there exists $\sigma_{0} > 1$ such that
	\[
	\|Df(x)^{-1}\| < \sigma_{0} \,\, \text{for every} \,\, x \in M \backslash V;
	\] 
	
	\item $f$ is \textit{not too contracting on} $V$: there is some small $\delta > 0$ such that
	\[
	\|Df(x)^{-1}\| < 1 + \delta \,\, \text{for every} \,\, x \in V.
	\]
\end{itemize}

In \cite{A} it is shown that this class satisfy the condition for non-uniform expansion. In the following we show a Lemma from \cite{A} which proves that such maps are non-uniformly expanding with Lebesgue as a reference measure.
\begin{lemma}
	Let $B_{1},\dots, B_{k},B_{k+1}= V$ a partition of $M$ into domains such that $f$ is injective on $B_{j}, 1 \leq j \leq p+1$. There exists $\theta > 0$  such that the orbit of Lebesgue almost every point $x \in M$ spends a fraction $\theta$ of the time in $B_{1} \cap \dots B_{p}$, that is,
	\[
	\#\{0\leq j < n \mid f^{j}(x) \in B_{1} \cap \dots B_{p}\} \geq \theta n,
	\]
	for every large $n \in \mathbb{N}$.
\end{lemma}

\subsection{Open zooming systems from local diffeomorphisms} We can obtain an open zooming system such that the zooming set $\Lambda$ is disjoint from the hole $H$ ($\Lambda \cap H = \emptyset$) using a local diffeomorphism $f : M \to M$ which is non-uniformly expanding as in the subsection \ref{local}. 

Let the zooming set $\Lambda = \cap_{j=-\infty}^{\infty}f^{j}(M\backslash V)$ with positive Lebesgue measure $m$. Since $\Lambda \cap V = \emptyset$, we can take a zooming reference measure $\mu = m / m(\Lambda)$ which has a Jacobian with bounded distortion. The zooming set $\Lambda$ is disjoint from $V$ and we can take  the hole $H \subset V$ given by \cite{S}[Theorem A] (and $\Lambda \cap H = \emptyset$). This setup now allows us to apply Theorems \ref{B} and \ref{C} to obtain existence and finiteness of (open) equilibrium state (uniqueness afterwards). We observe that in the work \cite{ALP2}[Lemma 2.1](3) we have that the Lebesgue measure has a Jacobian with bounded distortion.

As a concrete example on the interval $[a,b]$, we can take a dynamically defined Cantor set with positive Lebesgue measure. It can be seen in \cite{PT}[Chapter 4] as an expanding map $g$ over a disjoint union of intervals $P = I_{1} \uplus \dots \uplus I_{k}$ onto the interval $[a,b]$, where $I_{j} \in [a,b]$ is a compact interval for every $1 \leq j \leq k$, that is, every interval $I_{j}$ is taken onto $[a,b]$. Outside the union $P$ we can define the map to have a measurable map $f:[a,b] \to [a,b]$. In \cite{PT} we see that the map $g$ has bounded distortion and we can extend it to the map $f$ preserving this property. The hole $H$ can be taken outside the union $P$.

\end{document}